\documentclass[12pt,oneside]{amsart}
\usepackage{amscd,a4,amssymb,amsmath}
\usepackage{typearea}
\usepackage{hyperref}
\usepackage[all]{xy}

\newcounter{lemmacounter}

\newtheorem{lemma}[lemmacounter]{Lemma}
\newtheorem{proposition}[lemmacounter]{Proposition}
\newtheorem{theorem}[lemmacounter]{Theorem}

\newcommand{\geodesic}[1]{{#1}^{\circ}}

\newcommand{\IG}{{\bf G}}
\newcommand{\IC}{{\bf C}}
\newcommand{\IR}{{\bf R}}

\newcommand{\IZ}{{\bf Z}}
\newcommand{\IN}{{\bf N}}
\newcommand{\IA}{{\bf A}}
\newcommand{\IF}{{\bf F}}

\newcommand{\IQ}{\mathbf{Q}}
\newcommand{\IQpbar}{\overline{\mathbf{Q}}_p}
\newcommand{\IQp}{\mathbf{Q}_p}

\newcommand{\mat}[2]{{\rm Mat}_{#1}({#2})}
\newcommand{\gl}[2]{{\rm GL}_{#1}({#2})}

\newcommand{\gal}[1]{{\rm Gal}({#1})}
\newcommand{\cg}[1]{{\rm Cl}_{#1}}
\newcommand{\decomp}[1]{{\mathcal D}({#1})}

\newcommand{\ord}{{\rm ord}}

\newcommand{\ssm}{\smallsetminus}

\newcommand{\hecke}[2]{T_{#2}({#1})}

\newcommand{\modfuncS}{\mathfrak{M}}
\newcommand{\modfunc}[2]{\modfuncS_{#2}({#1})}
\newcommand{\spec}[1]{{\rm Spec}({#1})}

\newcommand{\Tp}[1]{T_{p}({#1})}

\newcommand{\dual}[1]{{#1}^{\vee}}

\newcommand{\homS}{{\rm Hom}}
\renewcommand{\hom}[1]{\homS({#1})}

\newcommand{\ho}[1]{T({#1})}

\newcommand{\tr}[1]{{\rm Tr}({#1})}
\newcommand{\en}[1]{{\rm End}({#1})}
\newcommand{\Ok}[1]{\mathcal{O}_K}

\newcommand{\Height}[1]{H({#1})}

\newcommand{\generator}{\theta}

\newcommand{\atopx}[2]{\genfrac{}{}{0pt}{}{#1}{#2}}

\begin{document}
\title{The  Tate-Voloch Conjecture in a Power of a Modular Curve}
\author[Philipp Habegger]{P. Habegger}

\maketitle

\newcommand{\dist}[2]{{\rm dist}_{#2}({#1})}
\newcommand{\distp}[2]{{\rm dist}'_{#2}({#1})}
\newcommand{\prox}[2]{{\rm \lambda}_{#2}({#1})}

\begin{abstract}
 Let $p$ be a prime. 
Tate and Voloch proved that a point of finite order in the algebraic
torus cannot be $p$-adically too close to a fixed subvariety without
lying on it.
 The current work is motivated by the analogy between torsion
points on semi-abelian varieties and special or CM points on Shimura
varieties. We prove the analog of Tate and Voloch's result in a
power of the modular curve $Y(1)$ on replacing torsion points by
points corresponding to a product 
of elliptic curves with complex multiplication and ordinary reduction.
Moreover, we show that the assumption on ordinary reduction is
necessary.  
\end{abstract}

\section{Introduction}

Let $p$ be a fixed prime. 
 Tate and Voloch conjectured \cite{TateVoloch} that a torsion point in a
semi-abelian variety cannot be $p$-adically too close to a subvariety
without actually lying on it.  
They proved their conjecture
 when the semi-abelian variety is an algebraic
torus. Buium \cite{Buium} obtained related results in a more abstract
framework
and  Scanlon  later proved \cite{Scanlon:TV98,Scanlon:TV99}
the conjecture for  semi-abelian varieties. He used work of
Chatzidakis and Hrushovski on the model theory of difference fields
which also enabled Hrushovski's proof of the Manin-Mumford Conjecture. 

There is a
 well-established analogy between
torsion points on semi-abelian varieties and CM points on Shimura
varieties. It is reflected in the formal similarity between the
conjectures of Manin-Mumford and Andr\'e-Oort. 
The purpose of this paper is to begin investigating the
Tate-Voloch Conjecture from the modular point of view.
We will confine ourselves to a power of the modular curve $Y(1)$ which
is the coarse moduli space of elliptic curves. 
As a variety this is  the affine line. 

The $p$-adic absolute value $|\cdot|_p$ extends uniquely from the field
$\IQ_p$
 of $p$-adics numbers
to an algebraic closure $\IQpbar$ of $\IQ_p$ and then to
a completion $\IC_p$ of $\IQpbar$.
The ring of integers in $\IC_p$ will be denoted by $\mathcal{O}_p$.  
If not stated otherwise, we will consider $Y(1)^n$ as a scheme over the spectrum of  ${\IC_p}$. 

Let $Z\subset Y(1)^n$ be a Zariski closed subset 
with vanishing ideal $I \subset \IC_p[X_1,\ldots,X_n]$. 
We define the $p$-adic distance
of $x \in \mathcal{O}_p^n$ to $Z$ as 
\begin{equation}
\label{eq:distance}
  \dist{x,Z}{p} = \sup \{ |f(x)|_p;\,\, 
f\in I \cap \mathcal{O}_p[X_1,\ldots,X_n] \}. 
\end{equation}
Then $\dist{x,Z}{p}=0$ if and only if $x\in Z(\IC_p)$. 
We certainly have
 $|f(x)|_p\le 1$ and so $\dist{x,Z}{p}\le 1$.

The $j$-invariant of an elliptic curve defined over $\IC_p$ with complex
multiplication is an algebraic integer and therefore an element of
$\mathcal{O}_p$. 
We call a point of $Y(1)^n(\IC_p)$ a CM point if its
coordinates are $j$-invariants of elliptic curves with complex
multiplication. 
It is convenient to call CM points of $Y(1)$ singular moduli.
For example,  $0\in Y(1)(\IC_p)$ is a singular moduli
since it is the $j$-invariant of the elliptic curve with complex
multiplication by $\IZ[(\sqrt{-3}+1)/2]$. 

An elliptic curve over $\IC_p$ with complex multiplication has  good reduction at the maximal
ideal of $\mathcal{O}_p$.
A CM point  of $Y(1)^n$ is called ordinary if its coordinates 
 correspond
 to elliptic curves with ordinary reduction.

If not stated otherwise, a subvariety of $Y(1)^n$ is an irreducible closed
subvariety of $Y(1)^n$ defined over $\IC_p$.

\begin{theorem}
\label{thm:ordmodtv}
Let $X$  be a subvariety of $Y(1)^n$.  There
  exists $\epsilon > 0$ such that if
$x$ is an ordinary CM point of $Y(1)^n$
with $x\not\in X(\IC_p)$, then 
$\dist{x,X}{p}\ge\epsilon$. 
\end{theorem}

Pink and Roessler \cite{PR:2004} used Hrushovski's setup to prove the Manin-Mumford
Conjecture using only algebraic geometry.
In the same vein our argument avoids the model theory of difference
field employed by  Scanlon. 
But we still   rely on  a carefully chosen field automorphism coming from class
field theory that
carries the arithmetic information. 
Roughly speaking,  the
uniformity statements provided by model theory are replaced by an
effective  version of Hilbert's Nullstellensatz due to Koll\'ar
\cite{Kollar}.
As in  Scanlon's argument we reduce the proof of 
 Theorem \ref{thm:ordmodtv} to the case where $X$ is a special
 subvariety of $Y(1)^n$. In Section \ref{sec:hecke} we give a complete
 description 
 of all such special subvarieties.
To treat  special subvarieties
we will apply Serre-Tate theory for ordinary elliptic curves in
characteristic $p$. It enriches the formal
deformation space of an ordinary elliptic curve with the structure of a formal 
torus. 
De Jong and Noot's \cite{deJongNoot} characterization of 
ordinary CM points as points of finite order will also play an important
role.

Our approach retains a connection to model theory.
 Indeed, we need  recent results of Pila
 on the weakly special
subvarieties contained in $X$ \cite{Pila:AO} and on
the Zariski closure in
$Y(1)^n$ of a Hecke orbit \cite{Pila:ellipticsurfaces}. The latter extends to varieties over $\IC$ an earlier 
 theorem proved together with the author \cite{hp:beyondAO} if the
 Hecke orbit consists of algebraic elements. These results
 rely on a strategy initially proposed by Zannier to prove the
 Manin-Mumford Conjecture using a theorem of Pila and Wilkie on
 rational points of sets definable in an o-minimal structure.

The choice of distance function (\ref{eq:distance}) was in part for
convenience. Another natural choice would be 
\begin{equation*}
  \distp{x,Z}{p} = \inf \{|x-y|_p;\,\, y\in Z(\IC_p)\}
\end{equation*}
where  $|\cdot|_p$ denotes also the $p$-adic sup-norm on $\IC_p^n$.
Using the Taylor
expansion of an element  $f\in I\cap \mathcal{O}_p[X_1,\ldots,X_n]$
around  $y\in Z(\IC_p)$ 
together with the ultrametric triangle inequality
yields $|f(x)|_p \le |x-y|_p\max\{1,|x-y|_p\}^{\deg f-1}$
for $x\in \mathcal{O}_p^n$.
But $|f(x)|_p\le 1$ and therefore, 
$|f(x)|_p\le |x-y|_p$. 
Taking first the infimum over $y\in Z(\IC_p)$ and then the supremum
over the admissible $f$ yields
\begin{equation*}
  \dist{x,Z}{p}\le \distp{x,Z}{p}. 
\end{equation*}
Therefore, Theorem \ref{thm:ordmodtv} holds
for the alternative distance $\distp{x,Z}{p}$.

The connection of our result to the
 Tate-Voloch Conjecture in the semi-abelian case begs the question why
 we restrict ourselves to \emph{ordinary} CM points. 
The reason,  apparent by the proposition below, is that
 subvarieties can be approximated arbitrarily well $p$-adically by
general singular moduli.
More precisely, we show that already the zero-dimension
variety
$X = \{0\}$ is the  $p$-adic limit of a sequence
 of singular moduli $x$
corresponding to elliptic curves with supersingular reducation at a
place above $p$. 
We will bound the $p$-adic distance in terms of the discriminant $\Delta(x)<0$
of the endomorphism ring 
of an  elliptic curve attached to $x$.

\begin{proposition}
\label{prop:approximate}
There is a constant $c>0$
 with the following property. 
Let $p$ be an odd prime with $p\equiv 2 \mod 3$. 
  There exists a sequence $x_1,x_2,\ldots$ of non-zero singular
  moduli with $\Delta(x_n)$ a fundamental discriminant, 
$\lim_{n\rightarrow\infty}\Delta(x_n)=-\infty$,
and $|x_n|_p\le c |\Delta(x_n)|^{-1/2}$. 
\end{proposition}

The proof of this proposition relies on explicit computations in the
endomorphism ring of the supersingular elliptic curve
$y^2=x^3+1$ in characteristic $p$
 combined with 
ideas
of Gross and Zagier \cite{GrossZagier}. We also make explicit 
an old result of
Nagel \cite{Nagel} on square-free values of quadratic polynomials.

Although a uniform lower bound as Theorem \ref{thm:ordmodtv} 
is impossible for unrestricted CM points, we ask if 
a weaker bound holds true.
Let $X$ be as in the theorem. Does there
exist positive constants $c$ and $\lambda$ such that
any CM point $x=(x_1,\ldots,x_n)$  with $x\not\in X(\IC_p)$ 
satisfies
$\dist{x,X}{p}\ge c
\max\{|\Delta(x_1)|,\ldots|\Delta(x_n)|\}^{-\lambda}$?
Certainly, $\lambda \ge 1/2$ if such an inequality were true.

It is natural to ask if Theorem \ref{thm:ordmodtv} can be extended,
for example, to the coarse moduli space of principally
polarized abelian varieties of  fixed dimension. Our approach would
require  a variant of Pila's
result mentioned above in this context. 
From this point of view, it 
 would also be interesting to have a proof of our result 
that circumvents Pila's Theorem and relies on the model theory of
difference fields.

The paper is organized a follows.  In 
Section \ref{sec:auto} we use class field theory to construct 
the field automorphism alluded to above. Section \ref{sec:hecke} uses Pila's
results to describe subvarieties that are almost invariant under Hecke
orbits with sufficiently large level.
Proposition \ref{prop:ind3} in Section \ref{sec:nonspecial} is
 a weak version of Theorem \ref{thm:ordmodtv} that cannot yet account for
 special subvarieties. In Section \ref{sec:serretate} we
 review  aspects of
 Serre-Tate theory that are required for treating special subvarieties
and proving Theorem \ref{thm:ordmodtv}
 in Section \ref{sec:special}. Finally, Proposition \ref{prop:approximate} is proved
 in the appendix.

The author heartily thanks  Thomas Scanlon for  productive discussions and 
especially for pointing  him towards  Serre-Tate theory which is a
crucial ingredient in the work at hand.

\section{Finding a Good Galois Element}
\label{sec:auto}

We let $\ord_p:\IC_p\rightarrow\IQ\cup\{\infty\}$ denote the valuation
on $\IC_p$ normalized such that $\ord_p(p)=1$.
As usual $\IZ_p$ is the ring of $p$-adic integers. We write $\IN$ for
the set of positive integers.

\begin{lemma}
\label{lem:order}
Let $\gamma \in 1 + p\IZ_p$ and suppose $D \in \IN$.
Then
$\ord_p(\gamma^D-1) = \ord_p(D) + \ord_p(\gamma-1)$ if $p\ge 3$ 
and
 $\ord_2(\gamma^D-1) \le \ord_2(D)+ \ord_2(\gamma^2-1) -1$ if $p=2$.
\end{lemma}
\begin{proof}
First we suppose that
$p\ge 3$. By Proposition II.5.5 \cite{Neukirch}
the $p$-adic logarithm $\log:1+p\IZ_p\rightarrow p\IZ_p$ is a homomorphism
 with 
$\ord_p \log(1+z) = \ord_p(z) $ for $z\in p\IZ_p\ssm\{0\}$. 
We apply these facts to $z = \gamma^D-1$ and $z=\gamma-1$ and
obtain
\begin{equation*}
  \ord_p(\gamma^D-1) = \ord_p \log(\gamma^D) = 
\ord_p (D \log\gamma) = \ord_p(D) + \ord_p\log\gamma
= \ord_p(D) + \ord_p (\gamma-1).
\end{equation*}

The logarithm has similar properties for $p=2$ albeit with a smaller domain of convergence. 
Now $\log :1+4\IZ_2\rightarrow 4\IZ_2$ satisfies
$\ord_2 \log(1+z) =\ord_2(z)$ for $z\in 4\IZ_2\ssm\{0\}$.
If $\gamma \in 1+2\IZ_2$, then 
 $\gamma^2 \in 1 + 4\IZ_2$. 
After replacing $\gamma$ by $\gamma^2$, the
 same argument as for odd primes yields 
\begin{equation*}
  \ord_2(\gamma^{2D}-1) = \ord_2(D) + \ord_2(\gamma^2-1).
\end{equation*}
But $\ord_2(\gamma^D - 1)  = \ord_2(\gamma^{2D}-1) - \ord_2(\gamma^D +
1)$ and so $\ord_2(\gamma^D + 1)\ge 1$ completes the proof.
\end{proof}

\begin{lemma}
\label{lem:matrix}
Let $k_0\in \IN$ and $D \in\IN$ with
  $k_0 \ge 2 \ord_p(2D)$
and assume
 $A_1,\ldots,A_n\in \gl{2}{\IQ_p}$.
Then there exist $\alpha,\beta\in \IQ_p^\times$ and $e\in\IZ$ 
such that the matrices
  \begin{equation}
\label{eq:conj}
    B_i =  A_i^{-1}\left(
    \begin{array}{cc}
       p^{-e} \alpha^D & \\ & p^{-e} \beta^D
    \end{array}\right)A_i.
  \end{equation}
satisfy 
 the following properties.
\begin{enumerate}
\item [(i)] 
For $1\le i\le n$ the matrix
$B_i$ has coefficients in $\IZ_p$ and there is an $i$ with
$B_i \not\in p\mat{2}{\IZ_p}$.
\item[(ii)] We have
  \begin{equation*}
k_0 \le     \ord_p(p^{-2e}\alpha^D\beta^D) \le 3Dk_0.
  \end{equation*}
\end{enumerate}
\end{lemma}
\begin{proof}
The matrix (\ref{eq:conj}) is  invariant under replacing
 $A_i$  by
a scalar multiple of itself. So we may suppose that all $A_i$ lie in 
$\mat{2}{\IZ_p}$.

We set $\delta_i =
  \det{A_i}$ and write
  \begin{equation*}
    A_i = \left(
    \begin{array}{cc}
      a_i & b_i \\ c_i & d_i
    \end{array}
\right). 
  \end{equation*}
Then
\begin{equation*}
    B_i= A_i^{-1}\left(
    \begin{array}{cc}
      p^{-e} \alpha^D & \\ &p^{-e}\beta^D
    \end{array}\right)A_i
= \frac{1}{p^e \delta_i}\left(
\begin{array}{cc}
  a_i d_i (\alpha^D-\beta^D) + \delta_i \beta^D & b_id_i(\alpha^D-
\beta^D) \\
  -a_ic_i(\alpha^D-\beta^D) & -a_id_i(\alpha^D-\beta^D) + \delta_i
  \alpha^D
\end{array}\right)
\end{equation*}
where the values  $\alpha,\beta,$ and $e$ will be specified later
on.
We  define integers
  \begin{equation}
\label{eq:defineki}
    k_i = \ord_p(\delta_i)  - \min\{\ord_p(a_id_i),\ord_p(b_id_i),\ord_p(a_ic_i)\}.
  \end{equation}
  The proof is a case by case analysis depending on the value of
\begin{equation*}
  k = \max \{k_1,\ldots,k_n\} \in \IZ.
\end{equation*}

The first case is $k\le k_0$. Here we set
\begin{equation*}
  \alpha = 1,\quad \beta = p^{k_0},\quad \text{and}\quad
e = - \max\{0,k\}.
\end{equation*}
The valuation of any
\begin{equation}
\label{eq:fourvalues}
  \frac{a_id_i}{p^e\delta_i}(1-p^{k_0D}),
  \frac{b_id_i}{p^e\delta_i}(1-p^{k_0D}),
  -\frac{a_ic_i}{p^e\delta_i}(1-p^{k_0D}),
  -\frac{a_id_i}{p^e\delta_i}(1-p^{k_0D}),
\end{equation}
is at least $-k_i-e \ge -k-e \ge 0$. 
From this and since the additional terms
 $p^{k_0D} p^{-e}$ and $p^{-e}$
in the upper left and lower right entries of $B_i$
are in $\IZ_p$,
we deduce
$B_i\in \mat{2}{\IZ_p}$. 
If $k=k_i$, then one of the four elements (\ref{eq:fourvalues}) has
order precisely $-k-e$. If $k \ge 1$, then $-e=k\ge 1$ and said order
is $0$.
The corresponding entry of $B_i$ also has order $0$
and the additional terms are harmless.
Hence (i) holds for positive $k$.
 If $k\le 0$ then $e=0$. 
The 
 lower right entry of $B_i$ is $-a_id_i(1-p^{Dk_0 })/\delta_i + 1$.
It is not divisible by $p$ if $\ord_p(a_i d_i/\delta_i)\ge 1$. 
But $a_id_i/\delta\in \IZ_p$ and so otherwise it is a unit. But then 
$a_id_i(1-p^{Dk_0})/\delta_i + p^{D k_0}$, the upper
left entry of $B_i$,
is   not divisible by
$p$ as $Dk_0 \ge 1$. So (i) also holds if $k\le 0$.

To prove (ii) we note that
$\ord_p(p^{-2e}\alpha^D \beta^D) = 
-2e + Dk_0 = 2\max\{0,k\}+Dk_0\ge Dk_0$. But this order is at most
$3Dk_0$ because of $k\le k_0$.

The second case is $k\ge k_0 + 1$.
We set  values 
\begin{equation*}
  \alpha = p^{k_0} + p^{k},\quad \beta = p^{k_0},\quad\text{and}\quad
e = \ord_p(\alpha^D-\beta^D) - k.
\end{equation*}
Note that $\ord_p(\alpha) = \ord_p(\beta) = k_0$
and $\ord_p(\alpha^D-\beta^D) \ge \ord_p(\alpha-\beta)= k$. So $e\ge
0$ and
 \begin{alignat}1
\label{eq:ordlowerbound}
   \ord_p(a_id_i(\alpha^D-\beta^D)) -\ord_p\delta_i - e 
&= \ord_p(a_id_i) + k -\ord_p\delta_i \\
\nonumber
&\ge \min \{\ord_p(a_id_i),\ord_p(b_id_i),\ord_p(a_ic_i)\} + k_i -
   \ord_p\delta_i \\
\nonumber
&=0,
 \end{alignat}
where the final equality is (\ref{eq:defineki}).
A simple modification of this argument shows
\begin{equation}
\label{eq:ordlowerbound2}
  \ord_p(b_id_i(\alpha^D-\beta^D)) -\ord_p\delta_i - e \ge 0 
\quad\text{and}\quad
  \ord_p(a_ic_i(\alpha^D-\beta^D)) -\ord_p\delta_i - e \ge 0.
\end{equation}

We apply Lemma \ref{lem:order} to $\gamma = \alpha/\beta = 1+p^{k-k_0}$ 
and find
\begin{equation*}
  e = \ord_p(\alpha^D-\beta^D) -k \le D k_0 - k +
\ord_p(\gamma-1)+\ord_p(D) +\left\{
  \begin{array}{cc}
    0  &:\text{if $p\ge 3$,}\\
    \ord_2(\gamma+1) - 1 &:\text{if $p=2$}.
  \end{array}
\right.
\end{equation*}
We note that $\ord_p(\gamma-1) = k-k_0$ and
 $\ord_2(\gamma+1)=\ord_2(2^{k-k_0}+2)\le 2$ if $p=2$. 
All in all we obtain
 $ e\le Dk_0 - k_0 + \ord_p(2D)$
regardless of $p$.
The hypothesis $2\ord_p(2D)\le k_0$
implies
\begin{equation}
\label{eq:bounde}
e  \le Dk_0   - k_0/2.
\end{equation}
In particular, $Dk_0 - e \ge k_0/2 > 0$ and so 
the additional terms $\beta^Dp^{-e}$ and $\alpha^Dp^{-e}$
which appear in $B_i$ have positive order.
We have thus proved $B_i\in\mat{2}{\IZ_p}$.

However, one of the 3 inequalities in 
(\ref{eq:ordlowerbound}) and (\ref{eq:ordlowerbound2}) must be an
equality if $k_i=k$. The additional terms are again harmless since they
have positive order.
Therefore $B_i \not\in p\mat{2}{\IZ_p}$ 
for all $i$ with $k_i=k$. This conclude the proof of part (i).

The relevant order in  (ii) is
$-2e+\ord_p(\alpha^D\beta^D) = -2e + 2Dk_0 \le 2Dk_0$ since $e\ge 0$.
The lower bound follows from (\ref{eq:bounde}).
\end{proof}

A place $v$ of a number field $F$ is a non-trivial absolute value whose
restriction to $\IQ$ is either 
 the restricted complex absolute value or the $p$-adic absolute value. 
We call $v$ finite if it is an ultrametric absolute value and we call
$v$ infinite
otherwise. It is well-known that finite places are in natural
bijection with non-zero prime ideals of the ring of integers of $F$.
We let $F_v$ denote the completion of $F$ with respect to $v$.

Suppose for the moment that $L/F$ is a finite abelian extension of number
fields and let $w$ be a place of $L$ that extends $v$. 
Let $\IA^\times_F$ denote the id\`eles of $F$ and
 $(\cdot, L_w/F_v) : K_v^\times
\rightarrow \gal{L_w/F_v}$ is the  local norm residue symbol.
If $s=(s_v)_{v}\in \IA^\times_F$, then the global norm residue symbol
\begin{equation*}
  (s,L/F)  =  \prod_v (s_v, L_w/F_v) \in \gal{L/F}
\end{equation*}
is a product of the local norm residue symbols.
The decomposition group $\gal{L_w/F_v}$ is a subgroup of $\gal{L/F}$
and depends only on $v$  since $L/F$ is abelian. For brevity we write
$\decomp{v} = \gal{L_w/F_v}$.
By abuse of notation we will  write
$\gal{L_v/F_v}$ and $(\cdot,L_v/F_v)$ instead
of $\gal{L_w/F_v}$ and $(\cdot,L_w/F_v)$.

Suppose $\sigma$ is an automorphism of $L$ and let 
 $\mathfrak{p}$ be a prime ideal of the ring of
integers of $L$ corresponding to a finite place $v$.  Then $\sigma v$
is a finite place of $L$ that
corresponds to the prime ideal $\sigma(\mathfrak{p})$. 

We write $\tau$ for complex conjugation.

\begin{lemma}
\label{lem:idele}
  Let $x_1,\ldots,x_n\in \IC$ be   $j$-invariants 
of elliptic curves with complex multiplication by orders in 
imaginary quadratic fields $K_1,\ldots,K_n\subset\IC$, respectively.
We abbreviate $F = K_1\cdots K_n$ and $L=F(x_1,\ldots,x_n)$. 
We suppose that $p$ splits in all $K_i$ and let $v$ be a place of
$F$ above $p$. Then $v\not=\tau v$. For $\alpha,\beta\in\IQ_p^\times$ let 
$s$ be the id\`ele 
\begin{equation*}
   (\ldots,1,\underbrace{\alpha}_{v},\underbrace{\beta}_{\tau v},1,\ldots) \in \IA_F^\times.
\end{equation*}
Then the following properties hold.
\begin{enumerate}
\item [(i)]The extension $L/F$ is abelian.
\item [(ii)]  
For  each $1\le i\le n$ the restrictions $v|_{K_i}$ and $\tau v|_{K_i}$
  are distinct places of $K_i$. 
\item [(iii)] 
For any $1\le i\le n$ we have
\begin{equation*}
  (s,L/F)|_{K_i(x_i)} = 
   \Big((\ldots,1,\underbrace{\alpha}_{v|_{K_i}},\underbrace{\beta}_{\tau
    v|_{K_i}},1,\ldots),K_i(x_i)/K_i\Big).
\end{equation*}
\item[(iv)] 
We have  $(s,L/F) \in \decomp{v}$. 
\end{enumerate}

\end{lemma}
\begin{proof} 
The classical theory of complex multiplication  implies
that  each $K_i(x_i)/K_i$ is an abelian extension, cf. Chapter 10.3
\cite{Lang:elliptic}.
By Galois theory the extension $L/F$ is abelian and part (i) holds true.

Statement (ii) follows since $p$ splits in each $K_i$. In particular, $v\not=\tau v$.

Part (iii) follows
directly from Proposition VI.5.2 \cite{Neukirch} and (ii).

Now for part (iv).
By Lemma 9.3 \cite{Cox} the extension $K_i(x_i)/\IQ$ is Galois with 
group $\gal{K_i(x_i)/K_i}\rtimes \gal{K_i/\IQ}$ where complex conjugation
 acts as $\tau \eta\tau|_{K_i(x_i)}  = \eta^{-1}$ for $\eta\in\gal{K_i(x_i)/K_i}$. 
In particular, $\tau(K_i(x_i))=K_i(x_i)$ and so $\tau(L)=L$ since $L$ is generated by
the $K_i(x_i)$. 
For arbitrary $\sigma \in \gal{L/F}$ the
automorphisms $\tau \sigma \tau|_L$ and  $\sigma^{-1}$ coincide on $K_i(x_i)$ and hence
also on $L$.

We have
\begin{equation*}
  (s,L/F) = (\alpha,L_v/F_v)(\beta,L_{\tau v}/F_{\tau v})
\end{equation*}
where $(\alpha,L_v/F_v)\in \decomp{v}$ and
$(\beta,L_{\tau v}/F_{\tau v})\in \decomp{\tau v}$.
If $\sigma \in \decomp{\tau v}$, then $\sigma \tau|_L v = \tau|_L v$
and hence 
 $\tau\sigma^{-1} v = \tau|_L v$ since $\tau\sigma\tau|_L=\sigma^{-1}$. 
Eliminating $\tau$ yields $\sigma\in \decomp{v}$
and therefore $(\beta,L_{\tau v}/F_{\tau v})\in\decomp{v}$.
This completes the proof.
\end{proof}

If $N=(N_1,\ldots,N_n)$ is a tuple of positive integers, then 
we write $T_N \subset Y(1)^n\times Y(1)^n$ for the 
correspondence that connects  the two $i$-th
coordinates in each copy of $Y(1)^n$  by a  cyclic isogeny of 
degree $N_i$.
This correspondence can be expressed  explicitly using modular
polynomials $\Phi_1,\Phi_2,\Phi_3,\ldots\in \IZ[X,Y]$, cf. Chapter 5
\cite{Lang:elliptic} for properties.
More precisely, 
\begin{equation*}
  T_N = \{(x_1,\ldots,x_n,y_1,\ldots,y_n) \in Y(1)^n\times Y(1)^n;\,\,
\Phi_{N_1}(x_1,y_1)=\cdots = \Phi_{N_n}(x_n,y_n)=0 \}.
\end{equation*}
We call $T_N$ the Hecke correspondence of level $N$. 
By Theorem 3, Chapter 5.2 \cite{Lang:elliptic} we have 
$\deg_Y\Phi_N = \deg_X \Phi_N = \Psi(N)$ with 
\begin{equation}
\label{eq:modulardegree}
  \Psi(N) =N\prod_{\ell|N}\frac{\ell+1}{\ell}
\end{equation}
where the product runs over primes $\ell$.

Let $\pi_{1},\pi_2:Y(1)^n\times Y(1)^n\rightarrow Y(1)^n$ denote the
projections to the first and last tuple of $n$ coordinates. 

The modular polynomials are monic considered as polynomials in $X$ or
$Y$. Thus $\pi_1|_{T_N}:T_N \rightarrow Y(1)^n$ is a finite morphism
and
if $Z\subset Y(1)^n$ is Zariski closed,
then so is
\begin{equation*}
  T_N(Z) = \pi_1(\pi_2^{-1}(Z) \cap T_N) \subset Y(1)^n.
\end{equation*}
If $x\in Y(1)^n(\IC)$, we abbreviate the finite set
$T_N(\{x\})$ by $T_N(x)$.

\begin{proposition}
\label{prop:goodgalois}
Let $x_1,\ldots,x_n\in\IC$ be $j$-invariants
  of elliptic curves with complex multiplication by orders in
  imaginary quadratic number fields $K_1,\ldots,K_n$, respectively. 
Let $F= K_1\cdots K_n$ and $L=F(x_1,\ldots,x_n)$, then $L/F$
is abelian by the previous lemma.
Suppose $p$ splits in all $K_i$
and $v$ is a place of $F$ above $p$. 
Let
$k_0\in\IN$ and $D\in\IN$ satisfy
$k_0 \ge 2\ord_p(2D)$. There exist integers $k_1,\ldots,k_n\ge 0$ and $\sigma \in
\gal{L/F}$ with the following properties.
\begin{enumerate}
\item [(i)] We have $k_0\le \max\{k_1,\ldots,k_n\} \le 3Dk_0$. 
\item[(ii)] We have $\sigma \in \decomp{v}$. 
\item[(iii)] We have $\sigma^D(x_1,\ldots,x_n) \in
\hecke{x_1,\ldots,x_n}{(p^{k_1},\ldots,p^{k_n})}$.
\end{enumerate}
\end{proposition}
\begin{proof}
For brevity we write $H_i = K_i(x_i)$ and $v_i = v|_{K_i}$.
Recall that $H_i/K_i$ is an abelian extension.

For any $1\le i\le n$ let $E_i$  be an elliptic curve over $\IC$ with
$j$-invariant $x_i$. 
There is a lattice $\Omega_i$  inside $K_i$ such that 
 $E(\IC)$ and $\IC / \Omega_i$ are isomorphic complex tori. 
Since $p$ splits in each $K_i$  there
 are two distinct embeddings $K_i\hookrightarrow \IQ_p$, the first
 determined by $v_i$ and the second one by
$\tau v_i$. 

If $K$ is a field extension of $\IQ$ and $\ell$ a prime we write
$K_\ell = K\otimes_\IQ \IQ_\ell$. 
If $M$ is a $\IZ$-module then we write
$M_\ell= M\otimes_\IZ \IZ_\ell$.

 We may identify $(K_i)_p$ with
 $\IQ_p^2$ as $\IQ_p$-algebras using the two embeddings mentioned above. Elements of $(K_i)_p$
 will be represented by column vectors in $\IQ_p^2$.
We may identify
 $(\Omega_i)_p$ with a free $\IZ_p$-module of rank $2$
 inside $\IQ_p^2$. 
We fix  a $\IZ$-Basis of $\Omega_i$ and arrange its image in $\IQ_p^2$ as the columns
 of a matrix $A_i\in \gl{2}{\IQ_p}$.

  We apply Lemma \ref{lem:matrix} to the $A_i,k_0,$ and $D$
 in order to obtain $\alpha,\beta\in \IQ_p^\times$ and $e$.

Let $s$ be the id\`ele as in Lemma \ref{lem:idele} and
\begin{equation*}
s' =   (\ldots,p^{-e},p^{-e},p^{-e},\ldots) \in \IA_F^\times,
\end{equation*}
this is a principal id\`ele. 
We set
\begin{equation*}
  \sigma =  (s,L/F)^{-1} 
\end{equation*}
and use Artin reciprocity to deduce
\begin{equation*}
  \sigma^D = (s' s^D,L/F)^{-1}.
\end{equation*}
Part (ii) of the current lemma follows from Lemma \ref{lem:idele}(iv).

Say $1\le i\le n$. We abbreviate $A=A_i, H=H_i$, $K=K_i$, $x=x_i$,  and $v=v_i$.

Let $E^{\sigma^D}$ be an elliptic curve  over $\IC$ with
$j$-invariant $\sigma^D(x)$. 
We now apply 
the
Main Theorem of Complex Multiplication, Theorem 3 in Chapter 10.2
\cite{Lang:elliptic}, to  $E^{\sigma^D}$.
Indeed,
$E^{\sigma^D}(\IC)$  and $\IC / s_K \Omega$ are isomorphic
complex tori where $s_K\in \IA^\times_K$ is the id\`ele norm of
$s' s^D$.

The lattice $\Lambda = s_K\Omega$ is determined locally by
$\Lambda_\ell$ at all
primes $\ell$ as follows. If $\ell\not=p$, then $\Lambda_\ell
= \Omega_\ell$ since the id\`ele $s_K$ has entry $p^e$ at
all places above $\ell$.
 However, things are different when $\ell=p$. 
Since the columns of $A$
are a $\IZ_p$ basis of $\Omega_p$,
 the columns of 
\begin{equation*}
  \left(
  \begin{array}{ll}
    p^{-e}\alpha^D & \\ & p^{-e}\beta^D
  \end{array}\right)A
\end{equation*}
constitute  
 a $\IZ_p$-basis of
 $\Lambda_p$.
By the first statement  of (i) in Lemma \ref{lem:matrix} we have
$\Lambda_p \subset \Omega_p$.
The Elementary Divisor Theorem implies
 $[\Omega_p:\Lambda_p] = p^{\ord_p(p^{-2e}\alpha^D \beta^D)}$
 and that
$\Omega_p/\Lambda_p \cong \IZ/p^{a}\IZ\times \IZ/p^b\IZ$
for integers $a\ge b$. 
Moreover,  $b=0$ if
\begin{equation}
\label{eq:cyclic}
 A^{-1}\left(
\begin{array}{ll}
      p^{-e}\alpha^D & \\ & p^{-e}\beta^D
\end{array}
\right)A \not\in p\mat{2}{\IZ_p}.
\end{equation}
By the Main Theorem of Complex Multiplication, the natural map
 $\IC/\Lambda\rightarrow\IC/\Omega$ 
restricted to the respective torsion groups is given by
$K/\Lambda \rightarrow K/\Omega$. 
We obtain a commutative diagram where the horizontal maps are 
the natural group isomorphisms
\begin{equation*}
 \xymatrix
{
\ar[d] K/\Lambda \ar[r]  &  \bigoplus_\ell K_\ell / \Lambda_\ell
\ar[d] \\   
K/\Omega   \ar[r]&  \bigoplus_\ell K_\ell / \Omega_\ell
}
\end{equation*}
Thus the kernel of $K/\Lambda\rightarrow K/\Omega$ is isomorphic
to $\Omega_p/\Lambda_p$.
Therefore, $\IC/\Lambda\rightarrow\IC/\Omega$ 
is an isogeny of degree
$p^{a+b}$. It is not cyclic if $b>0$. 
This yields part (iii) of the proposition.
Of course $k$ depends on $i$. We remark that
\begin{equation*}
k=  a-b \le a +b = \ord_p(p^{-2e}\alpha^D\beta^D) \le 3Dk_0
\end{equation*}
by (ii) of Lemma \ref{lem:matrix}. So the
upper bound for $\max\{k_1,\ldots,k_n\}$ in (i) of the assertion holds
true. 

To complete the proof of the proposition
we need to prove the lower bound as well.
Now  $\Omega_p/\Lambda_p$ is cyclic if
(\ref{eq:cyclic}) holds true. 
But this is the case for some $A=A_i$ by
the second statement of (i) in Lemma \ref{lem:matrix}. For such an $i$
we have $k_i=\ord_p(p^{-2e}\alpha\beta) \ge k_0$. This completes the proof of
(i).
\end{proof}

\section{Hecke Translates}
\label{sec:hecke}

We briefly recall the definition of weakly  special,  special, and
strongly special
 subvarieties of $Y(1)^n$.

Let $S_0,\ldots,S_r$ be a partition of $\{1,\ldots,n\}$ where $S_0$
may be empty but the $S_1,\ldots, S_r$ are non-empty.
Let $s_j = \#S_j$ and let us write
$S_j = \{i_{j1},\ldots,i_{js_j}\}$ with 
$i_{j1}<i_{j2}<\cdots < i_{js_j}$. 
A weakly special subvariety of $Y(1)^n$ is an irreducible Zariski
closed set $S$ 
determined by
\begin{equation*}
   \Phi_{N_{jk}}(x_{i_{j1}},x_{i_{jk}}) = 0
\text{ for all $1\le j\le r$, $2 \le k\le s_j$  and $x_i=c_i$ for all
  $i\in S_0$}.
\end{equation*}
where $N_{ij}$ are positive integers
and $c_i$ are fixed elements of $\IC_p$ (or $\IC$) for all $i\in S_0$. 

If all $c_i$ are singular moduli then we call $S$ 
 a special subvariety of $Y(1)^n$.

If $S_0=\emptyset$ then we call $S$ a strongly special subvariety of
$Y(1)^n$.

To ease notation we often omit the suffix ``of $Y(1)^n$'' when
speaking of weakly special, special, or strongly special subvarieties
of $Y(1)^n$. 

 Edixhoven \cite{Edixhoven:05} showed that this definition of special
 subvariety is
 equivalent of   one that is more natural from the point
 of view of Shimura varieties. 

  Let $X\subset Y(1)^n$ be a  subvariety. 
We define $\geodesic{X}$ to be the union $\bigcup_{Z\subset X} Z$
where  $Z$
runs over all weakly special subvarieties  with $\dim Z \ge 1$ that
are contained in $X$. 

The following proposition is a quick consequence of Pila's Structure
Theorem on weakly special subvarieties.

\begin{proposition}
\label{prop:geodesic}
  Let $X$ be an irreducible closed subvariety of $Y(1)^n$  both taken as
 over $\IC$. 
  There are finitely many strongly special subvarieties $S_i \subset
  Y(1)^{n-n_i}$  $(1\le i\le s)$ with $\dim S_i \ge 1$ and $0\le n_i\le n-1$,  subvarieties $W_i \subset
  Y(1)^{n_i}$, and coordinate permutations $\rho_i :Y(1)^n\rightarrow
  Y(1)^n$ such that
  \begin{equation*}
    \geodesic{X} = \bigcup_{i=1}^s \rho_i(S_i\times W_i). 
  \end{equation*}
\end{proposition}
\begin{proof}
  Any weakly special subvariety of $Y(1)^n$ of positive dimension
 is of the form
$\rho(S\times \{x'\})$ with $S\subset Y(1)^{n-n'}$ strongly special,
 $\rho$ a coordinate permutation,  
and $x'$ a point. The case $n'=0$ is possible. 

If for some choice of $x'$ the
 weakly special subvariety $\rho(S\times \{x'\})$ is
maximal among all weakly special subvarieties contained in $X$, then 
 $S$ must come from a finite set depending only on $X$
by Proposition 13.1 \cite{Pila:AO}.

The theorem on the dimension of a fiber of morphism, cf.  Exercise
 II.3.22
\cite{Hartshorne}, implies
that the condition  $x'\in Y(1)^{n'}(\IC)$ with $\rho(S\times\{x'\})\subset X$
determines a
 Zariski closed set for a fixed $S$. 
The proposition follows as we may restrict to finitely many $S$.
\end{proof}

The  full Hecke orbit of a point $x\in Y(1)^n(\IC)$ is 
\begin{equation*}
\ho{x}=  \bigcup_{N\in\IN^n} T_N(x).
\end{equation*}

\begin{theorem}[Pila]
\label{thm:modularmordell}
 Let $X$ be an irreducible closed subvariety of $Y(1)^n$ both taken as
 over $\IC$. If $x\in X(\IC)$   
then $(X\ssm\geodesic{X})\cap \ho{x}$ is finite. 
\end{theorem}
\begin{proof}
This is an immediate consequence of Theorem 6.2
\cite{Pila:ellipticsurfaces} which holds for varieties over $\IC$.
\end{proof}

We say that a subvariety of $Y(1)^n$
 has a special factor if, after possibly permuting
coordinates, it is of the form $S\times W$ with
\begin{equation*}
  S\subset Y(1)^{n-n'} \quad\text{a special subvariety where $n-n'\ge 1$}
\end{equation*}
and $W\subset Y(1)^{n'}$ Zariski closed. 

 Theorem \ref{thm:modularmordell} is used to prove the  following
proposition.

\begin{proposition}
\label{prop:TNX}
   Let $X\subset Y(1)^n$ be a subvariety  without a special factor.
Then there exists $N_0$ such that 
\begin{equation}
\label{eq:TNX}
  \text{if $N\in \IN^n$ and $|N|\ge N_0$, then } X\not\subset T_N(X).
\end{equation}
\end{proposition}
Before we come to the proof we briefly discuss the connection to 
 Edixhoven's  Theorem 4.1 \cite{Edixhoven:05} which, under a suitable
 restriction,
draws  a similar conclusion. The restriction 
requires all coordinates of $N$ to be equal.
More importantly, the prime divisors of the level need  to be large
with respect to $X$. In our application the coordinates  are powers of
the fixed prime $p$ which is unrelated to $X$.

\begin{proof}[Proof of Proposition \ref{prop:TNX}]
By the Lefschetz principle it suffices to prove the proposition
for an irreducible closed subvariety $X$ of $Y(1)^n$ defined over
$\IC$.

What happens if,  after permuting coordinates, we have $X= 
X'\times \{x''\}$,
  no coordinate function is constant on $X'\subset Y(1)^{n'}$, and
 $n'\le n-1$?
Then  no coordinate of $x''$ is a singular moduli, since $X$ would have a
special factor otherwise.
There cannot be two cyclic isogenies of different degree
between a pair of elliptic curves without complex multiplication.
So  $x''\in T_{N''}(x'')$ only if $N''=(1,\ldots,1)$. Since
$T_N(X)=  T_{N'}(X')\times T_{N''}(x'')$ for any $N\in\IN^n$ it is enough
 to verify  the proposition for $X=X'$ and $\dim X\ge 1$. 

The hypothesis and Proposition \ref{prop:geodesic} imply that 
$Z = X\ssm \geodesic{X}$ is a Zariski closed subset of $X$
with $\dim Z \le \dim X -1$.

We aim at a contradiction and suppose
\begin{equation}
\label{eq:inclusion}
  X \subset T_{N_i}(X) \quad\text{for infinitely many}\quad
N_i \in\IN^n. 
\end{equation}
We define
\begin{equation*}
  \Sigma_1 = \bigcup_{i\ge 1} X(\IC)\cap T_{N_i}(Z)(\IC)
\end{equation*}
and for good measure
\begin{equation*}
  \Sigma_2 = \{(x_1,\ldots,x_n)\in X(\IC);\,\, \text{some $x_i$ is
    a singular moduli}\}. 
\end{equation*}

Then $\Sigma_1$ and $\Sigma_2$ are both
 countable unions of Zariski closed
subsets of $X$, each  of 
dimension $\le \dim X - 1$. Indeed, $\dim X\cap T_{N_i}(Z)\le \dim
T_{N_i}(Z) = \dim Z \le \dim X - 1$ proves the claim for $\Sigma_1$.
We recall that
 no coordinate function is constant on $X$ and that
 there are only countably many CM points. So the claim holds for
 $\Sigma_2$ too. 

We find, by  the Baire Category Theorem for example, that
$\Sigma_1\cup \Sigma_2 \subsetneq X$. 
We fix an auxiliary point $x\in X(\IC) \ssm (\Sigma_1\cup \Sigma_2)$. 

After permuting coordinates we may suppose that if $N_{i1}$ is
the first coordinate of $N_i$, then $N_{i1}\rightarrow\infty$. 
By (\ref{eq:inclusion}) there is $x^{(i)} \in X(\IC)$ with $(x,x^{(i)})\in
T_{N_i}(\IC)$. Each $x^{(i)}$ lies  in the full Hecke orbit  $\ho{x}$.
Recall that  the first coordinate of $x$ is not a singular moduli.
By evoking the argument  from above on cyclic isogenies between  elliptic
curves without complex multiplication we find that
 $x^{(1)},x^{(2)},\ldots$ is an infinite sequence. 

Some $x^{(i)}$ must be in $Z(\IC)$  by Theorem \ref{thm:modularmordell}.
But then 
$x\in T_{N_i}(x^{(i)}) \subset T_{N_i}(Z)(\IC)$, contradicting
$x\not\in\Sigma_1$.  
\end{proof}

\section{Approximating Non-Special Subvarieties}
\label{sec:nonspecial}

By continuity we may extend any element of
$\gal{\IQpbar/\IQp}$ uniquely to an automorphism of $\IC_p$.
So we will apply elements of $\gal{\IQpbar/\IQp}$ 
to $\IC_p$ without further comment.

\begin{lemma}[Approximation Lemma]
  Let $\alpha_1,\ldots,\alpha_N\in \mathcal{O}_p$
 and suppose $\delta\in (0,1]$. There exists 
$D\in\IN$ such that if $\sigma\in \gal{\IQpbar/\IQ_p}$, then
\begin{equation*}
  |\sigma^D(\alpha_i)-\alpha_i|_p \le \delta 
\quad\text{for}\quad 1\le i\le N.
\end{equation*}
\end{lemma}
\begin{proof}
 Since $\IQpbar$
lies dense in $\IC_p$ there exists a finite Galois extension $K$ of $\IQ_p$
containing $x_1,\ldots,x_N$
with $|x_i-\alpha_i|_p\le \delta$ for $1\le i\le N$. 
Note that all $x_i$ must be integers since $\delta \le 1$.

Let $\mathcal{O}$ be the ring of integers of $K$ and
 $\pi$ a generator of its maximal ideal.
We fix the smallest integer $n\ge 0$ with
$|\pi^n|_p \le\delta$.
Any element of $\gal{\IQpbar/\IQ_p}$ induces an automorphism of the
finite ring $\mathcal{O} / \pi^n\mathcal{O}$.
If $D =  (\# \mathcal{O} / \pi^n\mathcal{O})!$
then $\sigma^D$ acts trivially
on $\mathcal{O} / \pi^n\mathcal{O}$ for all
$\sigma \in \gal{\IQpbar/\IQ_p}$. 
In other words, $\sigma^D(x_i) \in x_i + \pi^n\mathcal{O}$ or
\begin{equation}
\label{eq:conjbound}
  |\sigma^D(x_i)-x_i|_p \le |\pi^n|_p \le\delta.
\end{equation}

The ultrametric triangle inequality implies
\begin{alignat*}1
  |\sigma^D(\alpha_i)-\alpha_i| &=
|\sigma^D(x_i) - x_i + x_i - \alpha_i 
+ \sigma^D(\alpha_i) - \sigma^D(x_i) |_p \\
&\le \max\{|\sigma^D(x_i) - x_i|_p, |\alpha_i-x_i|_p,
|\sigma^D(\alpha_i-x_i) |_p\} \\
&\le \delta
\end{alignat*}
where we used $|\sigma^D(\alpha_i-x_i)|_p
=|\alpha_i-x_i|_p\le \delta$ and (\ref{eq:conjbound}).
\end{proof}

We extend $|\cdot|_p$  from $\IC_p$
 to the Gauss norm on $\IC_p[X_1,\ldots,X_n]$.

\begin{lemma}
\label{lem:idealmembership}
  Let $f_1,\ldots,f_N\in \IC_p[X_1,\ldots,X_n]$. There exists
 a constant $c=c(f_1,\ldots,f_N) > 0$ such that 
if $f$ lies in the ideal of $\IC_p[X_1,\ldots,X_N]$
 generated by $f_1,\ldots,f_N$ then there
are $a_1,\ldots,a_N\in \IC_p[X_1,\ldots,X_n]$ with
\begin{equation*}
  f = \sum_{i=1}^N a_i f_i
\end{equation*}
and
\begin{equation*}
  \max_{1\le i\le N} |a_i|_p \le c |f|_p. 
\end{equation*}
\end{lemma}
\begin{proof}
Stating that $f$ is in the ideal generated by
the $a_i$ is equivalent to stating that a certain inhomogeneous 
linear equation $Fx=f$ is solvable. Here $F$ is a matrix whose
entries are coefficients of the $f_i$ and $f$ is identified
with its coefficients suitably arranged as  a column vector. The
entries of $x$ are coefficients of polynomials $a_i$
such that $f=\sum_{i=1}^n a_i f_i$. 
After transforming $F$ into reduced row echelon form it is evident how
to 
 replace the $a_i$ by new polynomials satisfying the assertion.
\end{proof}

We collect some basic facts  on the $p$-adic distance function in 
the next lemma.

\begin{lemma}
\label{lem:distprop}
Let $Z$ be Zariski closed in $Y(1)^n$
and $Z'$ Zariski closed in $Y(1)^m$.
\begin{enumerate}
\item[(i)]
If $n=m$, then
\begin{equation*}
  \dist{x,Z}{p}\dist{x,Z'}{p}\le \dist{x,Z\cup Z'}{p}
\le \min \{\dist{x,Z}{p},\dist{x,Z'}{p}\}
\end{equation*}
for all $x\in \mathcal{O}_p^n$.
\item[(ii)]
There is a constant
$c=c(Z,Z')\ge 1$ such that if
 $x\in \mathcal{O}_p^n$ and $y\in \mathcal{O}_p^{m}$, then
\begin{equation*}
\dist{(x,y),Z\times Z'}{p}
\le c\max\{ \dist{x,Z}{p},\dist{y,Z'}{p}\}.
\end{equation*}
Moreover, 
\begin{equation*}
  \max\{\dist{x,Z}{p},\dist{y,Z'}{p}\}\le  \dist{(x,y),Z\times Z'}{p}.
\end{equation*}
\end{enumerate}
\end{lemma}
\begin{proof}
We recall that $I(Z\cup Z')=I(Z)\cap I(Z')$ which implies the second
inequality in (i). 
The first one  follows from $I(Z)I(Z')\subset I(Z\cup
Z')$.

We come to part (ii). The second inequality is immediate
as any element of $I(Z)$ or $I(Z')$ also vanishes on $Z\times Z'$ 
when considered as a polynomial in additional variables. 

The first inequality requires some care. 
Let $f_1,\ldots,f_{N}$ and $g_1,\ldots,g_{M}$ be generators
of the ideals $I(Z)$ and $I(Z')$, respectively. 
The first $N$ polynomials have variables $X_1,\ldots,X_n$ and
those of the second $M$ polynomials are $Y_1,\dots,Y_{m}$.
Without loss of generality, we suppose that the $f_i$ and $g_i$ have
coefficients in $\mathcal{O}_p$.

Let $f\in I(Z\times Z')$ have coefficients in $\mathcal{O}_p$.
The ideal $I(Z\times Z')$ is generated
by $f_1,\ldots,f_N$ and $g_1,\ldots,g_{M}$. 
By Lemma \ref{lem:idealmembership} we can find $a_i,b_i\in
\IC_p[X_1,\ldots,X_n,Y_1,\ldots,Y_{m}]$
with $f=\sum_{i=1}^N a_if_i + \sum_{i=1}^{M} b_i g_i$
and $\max_{i}\{|a_i|_p,|b_i|_p\} \le c |f|_p \le c$;
the constant  $c$ depends only on the $f_i$ and $g_i$ but not on
$f$. The ultrametric triangle inequality
implies 
\begin{equation*}
  |f(x,y)|_p \le\max_i \{|a_i|_p
|f_i(x)|_p,|b_i|_p|g_i(y)|_p\}
\le c\max\{|f_i(x)|_p,|g_i(y)|_p\}.
\end{equation*}
Therefore,
$|f(x,y)|_p \le c \max\{\dist{x,Z}{p},\dist{y,Z'}{p}\}$ and
part (ii) follows by taking the supremum over all admissible $f$.
\end{proof}

From now on, 
  let $X $ be a subvariety of $Y(1)^n$.

We now prove that an ordinary CM point that lies sufficiently close to
$X$ must lie close to one of finitely many special subvarieties
contained in $X$. 
We do this in 3 steps given by the following 3, increasingly refined, statements. 
Handling ordinary CM points that are close to a special subvariety
requires Serre-Tate theory and
will be postponed to the next section.

\begin{lemma}[Induction step - first version]
\label{prop:ind1}
 We suppose that $X$ does not have 
 a special factor. 
There exists a Zariski closed subset
$Z\subsetneq X$ 
with the following properties.
Let $\epsilon > 0$. There is a constant
$\delta=\delta_1(\epsilon,X) >0$ such that 
if $x\in Y(1)^n(\IC_p)$ is an ordinary CM point
\begin{equation*}
\text{with}\quad \dist{x,X}{p}\le \delta\quad\text{then}\quad
   \dist{x,Z}{p}\le \epsilon.
 \end{equation*}
\end{lemma}
\begin{proof}
We fix polynomials 
 $f_1,\ldots,f_N \in \IC_p[X_1,\ldots,X_n]$ 
that generate $I(X)$.
Without
loss of generality, we may suppose that the $f_i$ have coefficients in
$\mathcal{O}_p$. 
Later on we will use Koll\'ar's Sharp Effective Nullstellensatz
\cite{Kollar}.
Quadratic polynomials are not allowed in this result but this can be
amended by replacing $f_i$ by $f_i^2$ if necessary. 
Although the $f_i$ may no longer generate $I(X)$, their set of common
zeros is still $X$. This is the only property we will need.

The constant $\delta=\delta_1(\epsilon,X) > 0$ will be determined in 
due course.
 It will satisfy $\delta\le 1$. 

We apply the Approximation Lemma to $\delta$ and
all coefficients of all $f_i$ to obtain $D$ with
\begin{equation}
\label{eq:approxfi}
  |\sigma^D(f_i)-f_i|_p\le \delta
\end{equation}
for all $i$ and all $\sigma\in \gal{\IQpbar/\IQ_p}$.

We continue by choosing the least integer $k_0$ with $k_0\ge \max\{1,2\ord_p(2D)\}$
and
$p^{k_0} \ge N_0$  where $N_0$ is as in Proposition \ref{prop:TNX}.
So  $k_0$ satisfies the
hypothesis of Proposition \ref{prop:goodgalois}.
Moreover, if $N=(p^{k_1},\ldots,p^{k_n})$ with
$(k_1,\ldots,k_n)$ is as in
Proposition \ref{prop:goodgalois}(i),
 then 
$X\not\subset T_N(X)$ by (\ref{eq:TNX}). 
We define the finite union
\begin{equation*}
  Z =  \bigcup_{k_0\le k=\max\{k_1,\ldots,k_n\} \le 3Dk_0 }X\cap 
T_{(p_1^{k_1},\ldots,p_n^{k_n})}(X).
\end{equation*}
It is Zariski closed and satisfies $Z\subsetneq
X$.

Now suppose $x=(x_1,\ldots,x_n)\in \mathcal{O}_p^n$ is as in the assertion.

Proposition \ref{prop:goodgalois} gives us $\sigma$ and
$(k_1,\ldots,k_n)$ where $v$ is induced by restricting $|\cdot|_p$ to
$F$, we now take both $L$ and $F$ of this proposition as subfields of
$\IQpbar$.
 We extend $\sigma$ to an element of $\gal{\IQpbar/\IQp}$.

We introduce new independent variables $Y_1,\ldots,Y_n$
and consider the the collection of polynomials
\begin{equation}
\label{eq:polysystem}
  f_1(X_1,\ldots,X_n),\ldots, f_N(X_1,\ldots,X_n), \\
  f_1(Y_1,\ldots,Y_n),\ldots, 
f_N(Y_1,\ldots,Y_n),\\
  \Phi_{p^{k_1}}(X_1,Y_1),\ldots, \Phi_{p^{k_n}}(X_n,Y_n)
\end{equation}
involving the modular polynomials. 
We note that the modular polynomials cannot have degree $2$.

By Koll\'ar's Corollary 1.7 \cite{Kollar}  
there exists $s$,
bounded solely in terms of the $\deg f_i$ and $p^{k_i}$,
with the following property. 
The $s$-th power of any polynomial vanishing on the 
set of common zeros of (\ref{eq:polysystem}) in $\IC_p^{2n}$
is in the ideal generated by said polynomials.

We fix $f\in \mathcal{O}_p[X_1,\ldots,X_n]$ that vanishes on $Z$ 
with 
\begin{equation}
\label{eq:ex}
  |f(x)|_p  \ge \dist{x,Z}{p} - \delta^{1/s}.
\end{equation}
Let us keep in mind that $f$ depends on $x$. 

If $(x'_1,\ldots,x'_n,y'_1,\ldots, y'_n)$ is any common zero of (\ref{eq:polysystem}), then 
$(x'_1,\ldots,x'_n)\in Z(\IC_p)$. 
So  $f^s$ lies in the ideal generated by the polynomials
(\ref{eq:polysystem}).
We invoke Lemma \ref{lem:idealmembership} to 
find polynomials $a_{i},b_{i},c_{i} \in \IC_p[X_1,\ldots,Y_n]$ 
with
\begin{equation*}
  f^s = \left(\sum_{i} a_{i}f_i(X_1,\ldots,X_N)+b_{i}
  f_i(Y_1,\ldots,Y_N)\right) +
\sum_{i} c_{i} \Phi_{p^{k_i}}(X_i,Y_i)
\end{equation*}
and $\max_i\{|a_{i}|_p,|b_{i}|_p, |c_{i}|_p\} \le c |f^s|_p\le
c$; here $c$ depends on  the $f_i$ and the $p^{k_i}$ but not on $f$ or $x$.

Substituting $(X_1,\ldots,X_n)$ by  $x$ and 
$(Y_1,\ldots,Y_n)$ by $\sigma^D(x)$ makes the terms involving the modular
polynomials vanish. 
We now show that the  remaining terms are small $p$-adically. 
To start off, we use the ultrametric triangle inequality to show
\begin{equation*}
|f(x)|_p^s  
\le \max_{i} \{|a_{i}(x,\sigma^D(x)) f_i(x)|_p,
|b_{i}(x,\sigma^D(x))f_i(\sigma^D(x))|_p\}.
\end{equation*}
Now $|a_{i}(x,\sigma^D(x))|_p\le |a_{i}|_p$ 
since $x$ and $\sigma(x)$ have integral coordinates. 
A similar bound holds for the $b_{i}$.
We also remark that 
$|f_i(x)|_p\le \dist{x,X}{p}\le \delta$ by definition of the
distance. So
\begin{equation*}
  |f(x)|_p^s\le \max_i\{|a_{i}|_p,|b_{i}|_p\}  
\max_i\{\delta,|f_i(\sigma^D(x))|_p\}.
\end{equation*}

We may use (\ref{eq:approxfi}) to get rid of the remaining 
$|f_i(\sigma^D(x))|_p$.
Indeed,  $|f_i(\sigma^D(x))|_p = |(\sigma^{-D}f_i)(x)|_p$, so 
\begin{equation*}
  |f_i(\sigma^D(x))|_p = 
  |(\sigma^{-D}f_i)(x)-f_i(x)+f_i(x)|_p
\le \max\{|(\sigma^{-D}f_i-f_i)(x)|_p,|f_i(x)|_p\}
\end{equation*}
and thus
\begin{equation*}
  |f_i(\sigma^D(x))|_p \le
\max\{|(\sigma^{-D}f_i)-f_i|_p,|f_i(x)|_p\}
\le \delta.
\end{equation*}

We find
\begin{equation*}
  |f(x)|_p \le \max_{i}\{|a_{i}|_p,|b_{i}|_p\}^{1/s}
  \delta^{1/s} \le c^{1/s} \delta^{1/s}
\end{equation*}
and (\ref{eq:ex}) yields
  $\dist{x,Z}{p} 
\le (c^{1/s}+1)\delta^{1/s}$.
The proposition follows for any $\delta\in(0,1]$ with
$(c^{1/s}+1)\delta^{1/s} \le \epsilon$.
\end{proof}

We now generalize our statement 
 from subvarieties without special factors to subvarieties
that are not special. 

\begin{lemma}[Induction step - second version]
\label{prop:ind2}
 We suppose that $X$ is not a special subvariety
of $Y(1)^n$.  
There exists a Zariski closed subset
$Z\subsetneq X$  
with the following property.
Let $\epsilon > 0$. There is a constant $\delta = \delta_2(\epsilon,X) > 0$
such that if $x\in Y(1)^n(\IC_p)$ is an ordinary
CM point
\begin{equation*}
\text{with}\quad \dist{x,X}{p}\le \delta \quad\text{then}\quad
   \dist{x,Z}{p}\le \epsilon.
 \end{equation*}
\end{lemma}
\begin{proof}
If $X$ has no special factor, then the proposition
follows from Lemma \ref{prop:ind1}.

After permuting coordinates we may suppose that $X=S\times W$ with $S\subset Y(1)^{n'}$ special
and $W\subset Y(1)^{n''}$ a subvariety that has
no special factors. We remark that $n'\ge 1$ because $X$ has a special
factor and
 $n''\ge 1$ because $X$ is  not special.  
The first version of the induction step applied to $W$
yields a Zariski closed set
 $Z'\subsetneq W$. We remark that $Z'$ is independent of $\epsilon$.

First we apply the initial bound  of
   Lemma \ref{lem:distprop}(ii) 
to find $c\ge 1$  with
\begin{equation}
\label{eq:prod}
  \dist{(x',x''),S\times Z'}{p} \le c\max\{\dist{x',S}{p},\dist{x'',Z'}{p}\}
\end{equation}
for all $x'\in \mathcal{O}_p^{n'}$ and $x''\in \mathcal{O}_p^{n''}$.

We make the choice
 $\delta=\delta_2(\epsilon,X) = \min \{
 \delta_1(\epsilon/c,W),
 \epsilon/c\}$
and proceed to show that the Zariski
 closed set $Z=S\times Z'$ satisfies the assertion. 

Say $x=(x',x'')$ is as in the hypothesis.
If $\dist{x,X}{p}\le \delta$, then 
$\dist{x'',W}{p}\le \delta$
and $\dist{x',S}{p}\le \delta$
 by the second bound
of Lemma \ref{lem:distprop}(ii).
The previous
lemma provides
$\dist{x'',Z'}{p}\le \epsilon / c$. Recalling (\ref{eq:prod})
yields 
\begin{equation*}
  \dist{x,S\times Z'}{p} \le 
c \max\left\{\delta, \frac{\epsilon}{c}\right\}
\le  \epsilon.
\qedhere
\end{equation*}
\end{proof}

\begin{proposition}[Induction step - final version]
\label{prop:ind3}
  There exist $N\ge 0$ and a finite number of special subvarieties
$S_1,\ldots,S_N$ of $Y(1)^n$ contained in $X$ with the following
  property. Let $\epsilon > 0$. There is  a constant
$\delta=\delta_3(X,\epsilon)>0$ such that 
if $x\in Y(1)^n(\IC_p)$ is an  ordinary CM point
\begin{equation*}
\text{with}\quad \dist{x,X}{p}\le \delta \quad\text{then}\quad
   \dist{x,S_i}{p}\le \epsilon\quad\text{for some}\quad 1\le i\le N.
 \end{equation*}
\end{proposition}
\begin{proof}
Without loss of generality  $X$ is not special. 
We prove the proposition by induction on $\dim X$. 

Suppose $\dim X=0$. Since $X$ is not a CM point,
 the previous lemma
implies that $\dist{x,X}{p}$ is uniformly bounded from below for any
$x$ as in the assertion. This proposition follows  with $N=0$ for
$\delta_3(X,\epsilon) > 0$ sufficiently small.

Now we assume $\dim X\ge 1$. 

Let $\delta_2(\epsilon,X)$ be the constant from the previous lemma
and let $Z_1,\ldots,Z_{M}$ be the irreducible
components of the Zariski closed subset of $X$ it provides.
The $Z_i$ are independent of $\epsilon$.
 By induction we get special subvarieties $S_1,\ldots,S_N$
contained in the $Z_i$ and constants $\delta_3(Z_i,\epsilon)$. 

We now prove that $\delta=\delta_3(X,\epsilon) = \min_i
\{\delta_2(\delta_3(\epsilon,Z_i)^{M},X)\}$ is sufficient. 

Say $x$ is as in the assertion and $\dist{x,X}{p}\le \delta$. Then 
$\dist{x,Z}{p}\le\delta_3(\epsilon,Z_i)^M$ 
for all $i$ by Lemma \ref{prop:ind2}. 
 Lemma \ref{lem:distprop}(i) implies that
some component of $Z$, say $Z_1$, satisfies 
$\dist{x,Z_1}{p}^M \le \dist{x,Z}{p}$.
So 
\begin{equation*}
  \dist{x,Z_1}{p}\le \delta_3(\epsilon,Z_1). 
\end{equation*}
The induction hypothesis now guarantees
$\dist{x,S_i}{p}\le \epsilon$ for some $i$, as desired. 
\end{proof}

\section{A Brief Review of Serre-Tate Theory}
\label{sec:serretate}

We recall some consequences of Serre-Tate theory for an 
ordinary abelian
variety $A$ defined over 
 algebraically closed field $\kappa$ of characteristic $p>0$. 
Our applications will however be restricted to elliptic curves.

Let $R$ be an Artinian local ring with residue field $\kappa$. 
An admissible pair is a tuple $(\mathcal A,f)$ with $\mathcal{A}\rightarrow \spec{R}$ 
an abelian scheme and $f:A\rightarrow \mathcal{A}\otimes_R{\kappa}$ an
isomorphism of abelian varieties. 
Two admissible pairs $(\mathcal A,f),(\mathcal A',f')$ are called
equivalent if there exists an isomorphism
$\mathcal{A}\rightarrow\mathcal{A'}$ of abelian schemes over
$\spec{R}$ whose restriction to the special 
fiber composed with  $f$ is $f'$.
The  equivalence classes of admissible pairs form a set
$\modfunc{R}{A}$. The association
 $R\mapsto \modfunc{R}{A}$ is a functor from the category of
Artinian local rings to the category of sets.

  Let $M_R$ be the maximal ideal of $R$.
 We define the group
\begin{equation*}
  \widehat{\IG}_m(R) = 1 + M_R.
\end{equation*}
Note that there is an $n$ with $p^nM_R=0$ by Nakayama's Lemma. So the
abelian group $1+M_R$ is a $\IZ_p$-module. 
Moreover, $\widehat{\IG}_m$ is a functor  from   Artinian local rings with
residue field $\kappa$ to the category of $\IZ_p$-modules.

Since $A$ is ordinary, 
the subgroup $A[p^n] \subset A(\kappa)$ of elements of order
dividing $p^n$ is isomorphic to 
$(\IZ/p^n\IZ)^{\dim A}$ for all  $n\in \IN$. 
Let 
\begin{equation*}
  \Tp{A} = \varprojlim A[p^n]
\end{equation*}
 denote the
Tate module of $A$. It is a free $\IZ_p$-module of rank $\dim
A$.
Let $\dual{A}$ denote the dual abelian variety.

\begin{theorem}[Serre-Tate] 
There exists a  natural isomorphism
\begin{equation*}
  q: \modfuncS_A \rightarrow
    \hom{\Tp{A}\otimes\Tp{\dual{A}},\widehat{\IG}_m}
\end{equation*}
 of functors.
\end{theorem}
\begin{proof}
This is part of Theorem 2.1 \cite{Katz:serretate}.
\end{proof}

Now assume $K$ is a finite extension of the maximal unramified
extension of $\IQ_p$ in $\IQpbar$.
The ring of integers $\mathcal O$ of $K$ is a discrete valuation ring
with maximal ideal $M_{\mathcal O}$. The residue field
 $\kappa=\mathcal O/M_{\mathcal O}$ is an algebraic closure of $\IF_p$. 
 We set 
\begin{equation*}
   \modfunc{\mathcal O}{A} = \varprojlim \modfunc{\mathcal O/M_{\mathcal O}^{n+1}}{A}
\quad\text{and}\quad
\widehat{\IG}_m(\mathcal O) = \varprojlim\widehat{\IG}_m(\mathcal O/M_{\mathcal O}^{n+1}) = 
1+M_{\mathcal O}.
\end{equation*}
The former limit is to be understood as the set of 
 formal deformations of $A$. The latter is a $\IZ_p$-module. 

Let $\mathcal{A}\rightarrow \spec{\mathcal O}$ be an abelian scheme
and say $A\rightarrow
\mathcal{A}\otimes_{\mathcal O} \kappa$ is an isomorphism of abelian varieties. 
The direct system of schemes
$\mathcal{A}\otimes_{\mathcal O}(\mathcal{O}/M_{\mathcal O}^{n+1})$
yields
an element
\begin{equation*}
  q(\mathcal{A}) \in \varprojlim
\hom{\Tp{A}\otimes\Tp{\dual{A}},\widehat{\IG}_m(\mathcal O/M_{\mathcal O}^{n+1})}
=
\hom{\Tp{A}\otimes\Tp{\dual{A}},\widehat{\IG}_m(\mathcal O)}
\end{equation*}
where we take $\IZ_p$-module homomorphisms.
The choice of  a $\IZ_p$-basis of $\Tp{A}$ and of $\Tp{\dual{A}}$ induces
an isomorphism 
\begin{equation*}
  \hom{\Tp{A}\otimes\Tp{\dual{A}},\widehat{\IG}_m(\mathcal O)}\cong
\widehat{\IG}_m(\mathcal O)^{(\dim A)^2}.
\end{equation*}
Serre and Tate's Theorem enriches the set of formal deformations
of $A$ with the structure of an abelian group. 
De Jong and Noot characterized the elements of finite order.

\begin{theorem}[de Jong-Noot]
We keep the notation from above. 
Then $q(\mathcal{A})$ has finite order if and only if 
$\mathcal{A}\otimes_{\mathcal O} K$ has complex multiplication. 
\end{theorem}
\begin{proof}
  See Proposition 3.5 \cite{deJongNoot}.
\end{proof}

\section{Approximating Special Subvarieties}
\label{sec:special}

 We recall that $\Phi_N\in \IZ[X,Y]$ denotes the modular
polynomial of level $N\in\IN$.

\begin{lemma}
\label{lem:ordPhiN}
Suppose $N\in\IN$ and
  let $x_1,x_2\in\mathcal{O}_p$ be
ordinary singular moduli. If
\begin{equation*}
 \ord_p \Phi_N(x_1,x_2) > 6\frac{\Psi(N)}{p-1}
\quad\text{then}\quad \Phi_N(x_1,x_2)=0.
\end{equation*}
\end{lemma}
\begin{proof}
Let us consider  $P = \Phi_N(x_1,x_2+T) = a_0 + \cdots
+ a_dT^d \in \mathcal{O}_p [T]$. 
The polynomial $\Phi_N$ is monic of
 degree $\Psi(N)$  in $Y$, cf. (\ref{eq:modulardegree}). So 
 $P$ is monic and of degree $\Psi(N)$.
Let $t\in\IC_p$ be a root of $P$ with maximal order. Then
\begin{equation*}
\Psi(N) \ord_p(t) \ge  \ord_p(a_0) =\ord_p \Phi_N(x_1,x_2) > 6
\frac{\Psi(N)}{p-1}\quad\text{hence}\quad
\ord_p(t) > \frac{6}{p-1}. 
\end{equation*}

We have $\Phi_N(x_1,x'_2)=0$ with  $x'_2=x_2+t$.
In particular, $x'_2$ is a singular moduli and
\begin{equation}
\label{eq:j2j2p}
  \ord_p(x'_2-x_2) = \ord_p(t) > \frac{6}{p-1}
\end{equation}
implies that it is ordinary.

We first treat the case $p\not = 2$. Below $K$ will denote a
finite extension of the maximal unramified extension of  $\IQ_p$ to be specified during the argument
below. Let $\mathcal O$ be the ring of integers in $K$
and $M_\mathcal O$ the maximal ideal of $\mathcal O$. 

Let us fix a root $\lambda_2\in K$ of
\begin{equation*}
  2^8 (1-\lambda_2(1-\lambda_2))^3 - x_2 \lambda_2^2(1-\lambda_2)^2 = 0.
\end{equation*}
So $\lambda_2\in\mathcal{O}_p$
 since it is integral over $\mathcal{O}_p$. 
The ultrametric triangle inequality implies
$\ord_p \lambda_2 = 0$ and $\ord_p(1-\lambda_2)=0$. The equation 
\begin{equation*}
  y^2 = x(x-1)(x-\lambda_2)
\end{equation*}
defines an elliptic curve with $j$-invariant $x_2$. The
coefficients of this model are integers in $K$
and its discriminant  is
$2^4\lambda_2^2(\lambda_2-1)^2$. 
Thus our elliptic curve has good reduction. In other words, we
 obtain an elliptic scheme $\mathcal{E}_2\rightarrow\spec{\mathcal O}$
whose special fiber $E_2$ is an elliptic curve over the residue field
of $K$. 

After possibly increasing $K$ we find
$\lambda'_2\in K$ with 
\begin{equation*}
  2^8 (1-\lambda'_2(1-\lambda'_2))^3 - x'_2 {\lambda_2'}^2(1-\lambda'_2)^2 = 0
\end{equation*}
and
\begin{equation}
\label{eq:ordplambda}
  \ord_p(\lambda'_2-\lambda_2) \ge \frac 16 {\ord_p(x'_2-x_2)}>  \frac{1}{p-1}. 
\end{equation}
This elements gives us a second elliptic curve in Weierstrass form
\begin{equation*}
  y^2 = x(x-1)(x-\lambda'_2).
\end{equation*}
with $j$-invariant $x'_2$. As above we find that this curve has good
reduction and thus it yields an elliptic scheme
$\mathcal{E}'_2\rightarrow \spec{\mathcal O}$ with special fiber $E_2$. 

Note that  $\lambda'_2$ and $\lambda_2$ have equal
reduction. Therefore,  ${E_2}$ and 
${E'_2}$ are the same ordinary elliptic curve over the
residue field of $\mathcal O$. 
We fix a generator of 
$\Tp{{E_2}}$ and henceforth consider
the Serre-Tate parameter
as an element of $1+M_{\mathcal O}$.
Let $q_2$ and $q'_2$ be the 
Serre-Tate parameters of
$\mathcal{E}_2$ and $\mathcal{E}'_2$. 
Then $\mathcal{E}_2$ and $\mathcal{E}'_2$ are isomorphic modulo
$p^{\ord_p(\lambda_2-\lambda'_2)}$
and we obtain
\begin{equation}
\label{eq:ordp}
  \ord_p(q_2-q'_2) \ge \ord_p(\lambda_2-\lambda'_2) >
\frac{1}{p-1}
\end{equation}
from (\ref{eq:ordplambda}).

The case $p=2$ is  similar, but we  cannot rely on the
Legendre model as it necessarily leads to a curve of bad reduction.
In any case $\ord_2(x_2) = 0$ since 
an elliptic curve with $j$-invariant $0$ is supersingular in characteristic $2$. 
This and (\ref{eq:j2j2p}) entail $\ord_2(x'_2)=0$.
The $j$-invariant of the elliptic curve defined by
\begin{equation*}
  y^2+xy = x^3 - \frac{36}{x_2-1728}x - \frac{1}{x_2-1728}
\end{equation*}
is $x_2$. The coefficients involved are in $\mathcal O$ and
 the discriminant  equals
$x_2^2(x_2-1728)^{-3}$.
As before we get an elliptic scheme $\mathcal{E}_2$ with special fiber
$E_2$. 
And again
we introduce an  elliptic scheme $\mathcal{E}'_2$ with special fiber
$E_2$ and determined by
\begin{equation*}
  y^2+xy = x^3 - \frac{36}{x'_2-1728}x - \frac{1}{x'_2-1728}.
\end{equation*}
The generic fiber has
 $j$-invariant $x'_2$.
Just as in the case of odd characteristic  the
corresponding Serre-Tate parameters $q_2,q'_2$ satisfy
\begin{equation}
\label{eq:ord2}
  \ord_2(q_2-q'_2) \ge \ord_2(\lambda_2-\lambda'_2) > \frac{1}{p-1}=1.
\end{equation}

Now we suppose again that $p$ is arbitrary. 
We note that (\ref{eq:ordp}) and (\ref{eq:ord2}) both lead to 
\begin{equation}
\label{eq:ordpfinal}
  \ord_p(q_2-q'_2) > \frac{1}{p-1}.
\end{equation}

Since the generic fibers of 
 $\mathcal{E}_2$ and $\mathcal{E}'_2$ have complex
multiplication,  $\zeta = q'_2 q_2^{-1}$ is a 
root of unity
by de Jong and Noot's Theorem. 
Any root of unity in $1+M_\mathcal O$ has order $p^e$ for some integer $e\ge
0$ and it is classical that
\begin{equation*}
  \zeta=1 \quad\text{or}\quad \text{$\zeta\not=1$ and }\ord_p(1-\zeta)
=  \frac{1}{p^{e-1}(p-1)} \le \frac{1}{p-1},
\end{equation*}
cf. Lemma I.10.1 \cite{Neukirch}.

By (\ref{eq:ordpfinal})  we must have $\zeta=1$ and so $q_2=q'_2$. 
Therefore, for all $n\ge 0$ there is an isomorphism
$\mathcal{E}_2 \otimes \spec{\mathcal O / M_{\mathcal O}^{n+1}}
\rightarrow \mathcal{E}'_2\otimes \spec{\mathcal O / M_{\mathcal
    O}^{n+1}}$
and  these isomorphisms are compatible with the
morphisms
induced by the base change
$\spec{\mathcal O/M_{\mathcal O}^{n+1}}\rightarrow
\spec{\mathcal O/M_{\mathcal O}^{n+2}}$.
This entails that $\mathcal E_2$ and ${\mathcal E}'_2$
are isomorphic as schemes over $\spec{\mathcal O}$, cf. Scholie 5.4.2 \cite{EGAIIIa}. 
In particular,  the $j$-invariants $x_2$ and $x'_2$ 
of the respective generic fibers coincide. We deduce
$\Phi_N(x_1,x_2)=0$, as desired.
\end{proof}

Even the case $N=1$ of the previous lemma will play are role in our
application. So we formulate it separately.

\begin{lemma}
\label{lem:jdifference}
  Let $x_1,x_2\in\mathcal{O}_p$ be ordinary singular moduli. If
  \begin{equation*}
    \ord_p (x_1-x_2) > \frac{6}{p-1}\quad\text{then}\quad x_1=x_2. 
  \end{equation*}
\end{lemma}
\begin{proof}
  This follows from the previous lemma in the case $N=1$
since $\Phi_1 = X-Y$. 
\end{proof}

We now prove the main result of this section.

\begin{proposition}
  Let $S\subset Y(1)^n$ be a special subvariety. There exists a
  constant $\epsilon > 0$ with the following property. 
If  $x=(x_1,\ldots,x_n)\in \mathcal{O}_p^n$
is an ordinary CM point with $x\not\in S(\IC_p)$, then
\begin{equation*}
\dist{x,S}{p}\ge \epsilon.
\end{equation*}
\end{proposition}
\begin{proof}
Without loss of generality, we may suppose $S\not=Y(1)^n$. 
After permuting coordinates we have $S=S'\times S''$
where $S'\subset Y(1)^{n'}$ consists of a
 CM point and no coordinate function is constant
when restricted to $S''\subset Y(1)^{n''}$. 

Say $x=(x',x'')$ is as in the hypothesis
with $x'\in \mathcal{O}_p^{n'}$ and 
$x''\in \mathcal{O}_p^{n''}$.

We suppose first that $n'\ge 1$ and that $x'$ is not the CM point $S'$. 
We can use Lemma \ref{lem:jdifference} to obtain
$\dist{x',S'}{p} \ge p^{-6/(p-1)}$. Here the proposition follows from
Lemma \ref{lem:distprop}(ii).

If $n'=0$ or if  $x'$ is the CM point $S'$, then $S''$ is not a power of
$Y(1)$ and $x''\not\in S''(\IC_p)$. 
By the classification of special subvarieties at the beginning of
Section \ref{sec:hecke} there is  $N\in \IN$ 
in a finite set depending only on $S''$ and indices $n' < i<j\le n$
 with the following
properties. The polynomial
$\Phi_N(X_{i},X_{j})$ vanishes on $S''$ and  
$\Phi_N(x_i,x_{j})\not=0$ for the corresponding coordinates
$x_i,x_{j}$ of $x$. We now refer directly to Lemma \ref{lem:ordPhiN}
to obtain $\ord_p \Phi_N(x_i,x_{j}) \le 6 \Psi(N)/(p-1)$. Since
$\Phi_N(X_i,X_{j})$ is in the ideal of $X$ and has 
coefficients in $\IZ$, we deduce
$\dist{x,S}{p} \ge p^{-6\Psi(N)/(p-1)}$. 
\end{proof}

\begin{proof}[Proof of Theorem \ref{thm:ordmodtv}]
  Our theorem is now a direct consequence of the previous proposition
  and Proposition \ref{prop:ind3}.
\end{proof}

\appendix
\section{Highly Divisible Singular Moduli}

\subsection{Warming-up}

As a warming-up for the proof of  Proposition \ref{prop:approximate} 
we discuss a variation in the case $p=2$. This subsection is 
 conditional on a conjecture on prime
values of quadratic polynomials. Our main tool is a result of 
Gross and Zagier \cite{GrossZagier}.
We proceed by setting up some notation. 

If $K$ is a number field then $\mathcal{O}_K$ denotes its ring of
integers
and  $\cg{K}$ its class group.

Let $\ell \ge 5$ be a prime with
$\ell \equiv 3 \mod 8$. 
Then
 $\left(\frac{-\ell}{2}\right)=\left(\frac{2}{\ell}\right)=-1$ 
where $\left(\frac{\cdot}{\cdot}\right)$ is the Kronecker symbol. 
The imaginary quadratic
 field $K=\IQ(\sqrt{-\ell})$ has discriminant $-\ell\equiv 1
\mod 4$ and the prime
 $2$ is inert in $K$.
There is a unique singular moduli $j$ whose associated
elliptic curve has complex multiplication by 
$\mathcal{O}_K=\IZ[(\sqrt{-\ell}+1)/2]$ with  $j\in \IQ_2$.
We note that $j\not=0$ since $\ell\not=3$. 
The Hilbert class field of $K$ is $K(j)$.
Let $\mathcal{A}\in \cg{K}$, we will make a precise choice later on. 
For $x\in\IR$ we set
\begin{equation*}
  r_\mathcal{A}(x) = \#\{I\in \mathcal{A};\,\,\text{$I$  has norm $x$}\}.
\end{equation*}
We write  $\omega(n)$ for  the number of
distinct prime divisors of a positive integer $n$. 
Let $\sigma\in\gal{H/K}$ be the 
 image of $\mathcal{A}$
under the Artin homomorphism.
On taking the second singular moduli in Gross and Zagier's Proposition 3.8
 \cite{GrossZagier} to equal
$0$ we obtain
  \begin{equation}
\label{eq:ord2j}
  \ord_2(\sigma(j)) = \frac 32 \sum_{k\ge 1}\sum_{x\in\IZ}
2^{\omega(\gcd(2,x))} r_{\mathcal{A}^2}\left(\frac{3\ell-x^2}{2^{2+k}}\right).
  \end{equation}

Let $n$ be  odd and
 suppose that there exists $x\in\IZ$ with 
 \begin{equation}
\label{eq:schinzel}
   3x^2 + 2^{2+n}=\ell.
 \end{equation}
We remark that $\ell$ satisfies all conditions imposed further up. 

After multiplying by $3$ and rearranging we find
$(3\ell - (3x)^2)/2^{2+k} = 2^{n-k}.3$ where $k$ denotes an odd integer
with $1\le k\le n$.
Equation (\ref{eq:schinzel}) implies
$\left(\frac{3x^2}{\ell}\right)
=-\left(\frac{2}{\ell}\right)^{2+n} = - (-1)^{2+n} =1$
since $n$ is odd.
So $1=\left(\frac{3}{\ell}\right)=\left(\frac{-\ell}{3}\right)$
by quadratic reciprocity
and
$3$ splits in $K$. This means that $\mathcal{O}_K$ contains precisely
two ideals of norm
$3$.
By the theory of genera by Gauss and since the discriminant of $K$ is
a prime, $\cg{K}$ has no elements
of order $2$,  cf. Proposition 3.11 \cite{Cox}. So $\mathcal{A}\mapsto\mathcal{A}^2$ induces an
automorphism of $\cg{K}$. Hence we may choose $\mathcal{A}$
such that
 $\mathcal{A}^2$ contains an ideal of norm $3$. 
But $n$ and $k$ are both odd, so $2^{n-k}$ is a perfect square and therefore
$r_{\mathcal{A}^2}(2^{n-k}.3) \ge 1$.
We deduce 
\begin{equation*}
  r_{\mathcal{A}^2}\left(\frac{3\ell - (3x)^2}{2^{2+k}}\right)
\ge 1\quad\text{and hence}\quad 
\sum_{x\in\IZ}
2^{\omega(\gcd(2,x))}  r_{\mathcal{A}^2}\left(\frac{3\ell - x^2}{2^{2+k}}\right)\ge 2.
\end{equation*}

On summing over all odd $k$ with $1\le k\le n$ we
use (\ref{eq:ord2j}) to deduce
\begin{equation*}
  \ord_2(\sigma(j)) \ge \frac 32  (n+1).
\end{equation*}

So $\sigma(j)$ is non-zero and converges to $0$ in the $2$-adic
topology  if $n$ can be made
arbitrarily large. 

However, our construction only works for $n$ if 
there exists $x\in\IZ$ with (\ref{eq:schinzel}) a prime since
Gross and Zagier's result requires 
a prime discriminant.
 Schinzel's Hypothesis predicts that there are infinitely many
$x$ with (\ref{eq:schinzel}) a prime.
This conjecture is open and seems out of reach at the moment.

\subsection{Computations with Quaternions}
\label{sec:quaternion}
 Gross and Zagier's work provides  precise information on the
 $p$-adic valuation of differences of certain singular moduli. 
But cruder estimates are  sufficient for the proof of Proposition
\ref{prop:approximate}. These  will follow 
from very explicit calculations in a quaternion algebra and some ideas
of Gross and Zagier. 

Our quaternion algebras are all over $\IQ$ and as a reference we use
mainly the book of Vign\'eras \cite{Vigneras}. 
Let $K$
be the number field $\IQ(\sqrt{-3})$ and $\Ok{K}=\IZ[\generator]$ its ring of
integers where $\generator = (\sqrt{-3}+1)/2$. Below
$\alpha\mapsto\overline\alpha$ denotes the non-trivial automorphism of
$K\ni \alpha$. 
The different $\mathcal D$ of $K$ is the  ideal
$\sqrt{-3}\Ok{K}$. We define  $\mathfrak{q} =
(2+\sqrt{-3})\Ok{K}$ and remark that $\mathfrak{q}\overline{\mathfrak
  q}=7\Ok{K}$. 
Let us suppose that $p$ is as in Proposition \ref{prop:approximate}. 
Thus $p$ is inert in $K$ and odd. 

For $\alpha,\beta\in K$ we set
\begin{equation*}
  [\alpha,\beta] = \left(
  \begin{array}{cc}
    \alpha & \beta \\ -7p\overline\beta & \overline\alpha
  \end{array}\right)
\quad\text{and}\quad Q = \left\{[\alpha,\beta];\,\, \alpha,\beta\in K\right\}.
\end{equation*}
We  consider $K$ as a subset of $Q$ by virtue of $\alpha\mapsto [\alpha,0]$.
Then $Q$ is a ring because
 $u^2=-7p$ where $u=[0,1]$  and even a
 four dimension $\IQ$-algebra.  We have $u\alpha =
\overline\alpha u$ from which it is easy to deduce
 that $\IQ$ is the center of $Q$.
Moreover, $[\alpha,\beta] [\overline\alpha,-\beta]$ is $\alpha\overline\alpha +
7p\beta\overline \beta \in\IQ$ and vanishes if and only if
$[\alpha,\beta]=0$. 
Hence $Q$ is a quaternion algebra and a skew-field. 
It therefore ramifies at at least one finite place of $\IQ$. 
We will see momentarily that $p$ is such a place and that it is the
only one.

The subset 
\begin{equation*}
  \mathcal O = \left\{[\alpha,\beta];\,\, \alpha \in \mathcal{D}^{-1}, 
 \beta\in \mathfrak q^{-1}\mathcal{D}^{-1},\text{ and }
\alpha- 7\beta \in\Ok{K} \right\}\subset Q
\end{equation*}
is a free $\IZ$-module of rank $4$. 
A short calculation verifies that it is a sub-ring of $Q$. Thus
$\mathcal O$ is
an order of $Q$. 
It is readily checked that 
\begin{equation*}
  [1-2\generator,0], [1-\generator,0],
  \left[\frac{1-2\generator}{3},\frac{4-5\generator}{21}\right], 
\left[1-\generator,\frac{3-9\generator}{21}\right]
\end{equation*}
are elements of $\mathcal O$. Let us denote them by $b_1,\ldots,b_4$. 
In a moment we will prove
 that these elements constitute a $\IZ$-basis of $\mathcal O$,
that $\mathcal O$ is a maximal order, and that $Q$ ramifies
precisely at $p$ and $\infty$.

Indeed, the reduced trace of any element $[\alpha,\beta]\in Q$ is
$\tr{[\alpha,\beta]}=\alpha+\overline \alpha$ so the
 $4\times 4$ matrix with entries
$\tr{[ b_i b_j]}$ is
\begin{equation*}
  -\left(
  \begin{array}{cccc}
    6 & 3 & 2 & 3 \\
    3 & 1 & 1 & 1 \\
    2 & 1 & \frac{2p+2}{3} &  p+1 \\
    3 & 1 & p+1 & 2p+1
  \end{array}\right)
\end{equation*}
and has determinant $-p^2$. 
Thus $\{b_1,\ldots,b_4\}$ is linearly independent
and
 the reduced discriminant $d(\mathcal O)$ of
$\mathcal O$ is a divisor of $p$. But $d(\mathcal O)$ is
 a multiple
 of the product of all finite primes where $Q$ ramifies.
We have seen above that this product is not $1$.
So $d(\mathcal O)=p$ and the remaining claims hold as well.

We write $W\subset \IC_p$ for the ring of  integers in  the completion of the
 maximal unramified algebraic extension of $\IQ_p$.

If $E$ is an elliptic curve defined over a field then $\en{E}$
denotes the ring of its endomorphisms  defined over an algebraic
closure of the base field. 
Say $E$ is the elliptic curve determined by $y^2=x^3+1$ taken
as an equation with coefficients in $W$. It
has complex multiplications by $\Ok{K}$ and good  reduction $E_0$
modulo $pW$ since $p\ge 5$. Moreover, $E_0$ is supersingular because
$p$ is inert in $K$. 

The following statements are well-known facts from the theory of
elliptic curves.  
The ring  $\en{E_0}$  is a maximal
order in $\en{E_0}\otimes\IQ$ and  the latter is
a quaternion algebra which
 is ramified precisely at $p$ and $\infty$. 
Reduction modulo $pW$ induces an injective homomorphism from
$\en{E}$ to  $\en{E_0}$. We will thus take $K$ as a
sub-ring of $\en{E_0}\otimes \IQ$.

Up-to isomorphism there is only one quaternion algebra that is
ramified at $p$ and $\infty$. 
By the Skolem-Noether Theorem we may choose an
 isomorphism $\en{E_0}\otimes\IQ\rightarrow Q$ 
that is compatible with the two inclusions $K \subset \en{E_0}\otimes
\IQ$ and $K\subset Q$.
Since  $\mathcal O$ and $\en{E_0}$ both contain the principal ideal
domain $\Ok{K}$ and all three
are maximal orders, in their respective sense,
  Eichler's Theorem 4, Chapter 4 \cite{Eichler:Tata}
implies that $\mathcal O$ and $\en{E_0}$ are conjugate by an element
of $\Ok{K}$. 
 Henceforth we will identify $\en{E_0}\otimes
\IQ$ with $Q$ and $\en{E_0}$ with $\mathcal O$.  

This setup enables explicit calculations with endomorphisms
of $E_0$. 
 
\begin{lemma}
\label{lem:construct}
  Let $n,x\in \IZ$ with $n\ge 0$ and $x$ odd. Then 
    \begin{equation*}
      \varphi = \left[\frac 12 - \frac{2\generator -1}{2}x,
        \frac{3-2\generator}{7}p^{n}\right]
    \end{equation*}
is an element of $\Ok{K}+p^{n}\mathcal O$
which satisfies
$\varphi^2-\varphi + (1+d)/4=0$ where
$d = 3x^2+4p^{2n+1}$. 
\end{lemma}
\begin{proof}
  The proof is by direct verification. We abbreviate
  $\varphi=[\alpha,\beta]$ and remark that $\alpha=1/2-(2\generator-1)x/2 =
  (1-\sqrt{-3}x)/2\in \Ok{K}$ since $x$ is odd. 
Hence it suffices to verify $[0,(3-2\generator)/7]\in \mathcal O$.
But this follows from
$\mathcal{D}\mathfrak{q}(3-2\generator)/7 =\sqrt{-3}(2+\sqrt{-3})(3-2\generator) /7\Ok{K} = \sqrt{-3}\Ok{K}$.

To prove the second claim we remark that $\varphi^2 - T \varphi + N =
0$ where $T$ and $N$ are, respectively, the reduced trace and reduced norm of $\varphi$.
Now $T=\alpha+\overline\alpha = 1$ and
\begin{equation*}
  N = \alpha\overline\alpha + 7p\beta\overline \beta = 
\frac{1+3x^2}{4}+p^{2n+1}\frac{(3-2\generator)(3-2\overline\generator)}{7}
= \frac{1+3x^2}{4}+p^{2n+1} = \frac{1+d}{4}.\qedhere
\end{equation*}
\end{proof}

We remark that $d$ as in this lemma satisfies $-d\equiv 1 \mod
4$.

Let $n\ge 0$ be an integer and let $E$ be the elliptic curve from
above. Let
$E_n$ denote the elliptic scheme defined
by the base change  to   $\spec{W/p^{n+1} W}$  of the elliptic scheme
over $\spec{W}$ determined by 
 $y^2=x^3+1$.
The reduction operation implies that 
the  endomorphism rings  are filtered
as
\begin{equation*}
\Ok{K}\subset \cdots\subset  \en{E_n}\subset \cdots\subset\en{E_1}\subset \en{E_0}.
\end{equation*}

Let $\widehat E$ be the formal group law
attached to $E$ with respect to the model $y^2=x^3+1$. It is a power series in two variables and
coefficients in $W$. If we reduce this power series modulo
$p$ we obtain the formal group law $\widehat E_0$ of $E_0$. 

Suppose that $\varphi \in \en{E_0}$ and let $\widehat \varphi\in
W/pW[[T]]$  be the power series representing $\varphi$. 
Recall that multiplication-by-$p$ of $\widehat E$ is a power series
$[p]= p f(T) + g(T^p)$ where $f,g\in W[T]$ have no constant term.
If $\widehat \varphi'\in W[[T]]$ is any lift of $\widehat \varphi$, 
then a simple induction shows that $[p]^{n} \widehat \varphi'$ 
is well-defined modulo $p^{n+1}$. 
This implies that  
$p^{n}  \varphi$ is an endomorphism of $E_{n}$.
Therefore, $p^n\en{E_0}\subset \en{E_n}$.
Since elements of $\Ok{E}$ are endomorphisms of $E_n$  we find
\begin{equation}
\label{eq:endincl}
\Ok{K}+  p^{n} \en{E_0} \subset \en{E_{n}}
\end{equation}
for all $n\ge 0$.

The next lemma relies on
 Gross and Zagier's \cite{GrossZagier} version of Deuring's Lifting Theorem.

 \begin{lemma}
\label{lem:construct2}
   Let $n,x,$ and $d$ be as in Lemma \ref{lem:construct}.
We suppose in addition that $p\nmid d$ and that $d$ is square-free.
Then there exists a singular moduli $x_n\in W$ such that $(1+\sqrt{-d})/2$ is an endomorphism
of the associated elliptic curve and $|x_n|_p \le p^{-(n+1)}$. 
 \end{lemma}
 \begin{proof}
   Let $\varphi$ be the endomorphism from Lemma \ref{lem:construct}. 
It satisfies $\varphi^2 - \varphi + (1+d)/4=0$ and lies in 
$\en{E_n}$ by (\ref{eq:endincl}). 
Proposition 2.7 \cite{GrossZagier} provides an elliptic curve $E'$
whose endomorphism ring contains $(1+\sqrt{-d})/2$ and which is
isomorphic to $E$ modulo $p^{n+1}W$.  
We remark that Hensel's Lemma and $p\nmid d$ imply that the polynomial
$x^2 - x +(1+d)/4$ 
has precisely two roots in $W/p^{n+1}W$. The hypothesis of Proposition
2.7 requires that  $d$ is a  fundamental discriminant.
The $j$-invariant of $E$ is zero. So the $j$-invariant $x_n$ of $E'$
satisfies
$|x_n|_p\le p^{-(n+1)}$ by Proposition 2.3 \cite{GrossZagier}. 
 \end{proof}

In general $3x^2+4p^{2n+1}$ may have a  quadratic factor, eg.
$3+4\cdot 11^{2\cdot 7+1}\equiv 0 \mod 49$.  
So we must carefully choose $x$ in order to apply Lemma \ref{lem:construct2}.
By an old result of Nagel \cite{Nagel} the polynomial
  $3x^2+4p^{2n+1}$ attains infinitely many  square-free values at integer
arguments. If $n\ge 1$, then
any such value
 is  coprime to $p$ 
and its corresponding argument  is necessarily  odd.
So the hypotheses of Lemmas \ref{lem:construct} and
 \ref{lem:construct2} are satisfied.
We obtain a singular moduli $x_n\in W$ with $|x_n|_p\le p^{-(n+1)}$.
Moreover,  $x_n\not=0$ because $3\nmid |\Delta(x_n)| = 3x^2+4p^{2n+1}$.

\subsection{Square-free Values of a Quadratic Polynomial}

To prove Proposition \ref{prop:approximate} we  must
  bound  $|x_n|_p$ from the previous section
 in terms of $|\Delta(x_n)|$.
To this end we need  an upper bound for one square-free value 
of the polynomial
\begin{equation}
\label{eq:polyf}
  f = 3x^2+4p^{2n+1}
\end{equation}
 in terms of $p^{n+1}$.

 As Igor Shparlinski  pointed out to the author, 
 Iwaniec and Friedlander \cite{FI:Squarefree} recently gave estimates for
 squarefree values of  
quadratic, monic polynomials.  
To treat (\ref{eq:polyf}) we will make explicit a sieving method 
 presented in Chapter 1 of the same authors' book
\cite{FI:Opera}.

In this section we assume $p\ge 5$  and that  $n\in\IN$.

If $y \ge 0$ is a real number we define
\begin{equation*}
  N(y) = \#\{ x\in \IN;\,\,  f(x) \le y
    \quad\text{and}\quad f(x)\text{ is square-free}\}.
\end{equation*}

Let $\mu(\cdot)$ denote the M\"obius function. If $m$ is a positive integer then 
$\sum_{d^2 | m} \mu(d) = 1$ if and only if $m$ is square-free,
otherwise this sum vanishes. Thus
\begin{equation}
\label{eq:Ny}
  N(y) = \sum_{\atopx{x\ge 1}{f(x)\le y}}
\sum_{d^2| f(x)}\mu(d)
= \sum_{1\le d \le y^{1/2}} \mu(d) A_{d^2}(y)
\quad\text{with}\quad
  A_{d^2}(y)  = \sum_{\atopx{x\ge 1,\, f(x)\le y}{
d^2|f(x)}} 1.
\end{equation}

Let $\epsilon \in (0,1)$ be a constant to be determined in course of
  the argument below.
In the following we use Landau's big-$O$ notation. 
All implicit constants
are absolute and thus
 independent of $p, n,\epsilon $, and the parameter $y$.

\begin{lemma} 
Let $d$ be a positive integer
 and suppose $y \ge 4p^{2n+1}\epsilon^{-1}$.
 Then
    \begin{equation}
\label{eq:Adestimate}
A_{d^2}(y)= \sqrt{\frac y3}\frac{\rho(d^2)}{d^2} + O\left(\rho(d^2)+
\epsilon y^{1/2}\frac{\rho(d^2)}{d^2}\right)
    \end{equation}
where $\rho(m) = \# \{ b\in\IZ/m\IZ;\,\, f(b)\equiv 0 \mod m\}$ for any
integer $m\ge 1$. 
\end{lemma}
\begin{proof}
We split the sum
$A_{d^2}(y)=\sum_{\atopx{0\le a <d^2}{d^2|f(a)}} A_{d^2,a}(y)$ up into terms 
\begin{equation*}
 A_{d^2,a}(y) = \sum_{\atopx{x\ge 1,\, f(x)\le y}{
x\equiv a\mod d^2}} 1.
\end{equation*}
We note that
$A_{d^2,a}(y)$
only contributes to $A_{d^2}(y)$ if $a$ is a root of
 $f$ modulo $d^2$.
The lemma will follow from estimating
each $A_{d^2,a}(y)$ separately for integers $0\le a < d^2$
with $d^2|f(a)$. 

  If $k$ is an integer and if $x=a+d^2k\ge 1$ satisfies $3x^2+4p^{2n+1}\le y$, then 
$3d^4k^2\le y$ and so $k\le \sqrt{y/3}d^{-2}$.
Moreover, the bound  $1\le a+d^2 k < d^2(1+k)$
implies $k\ge 0$  and
so $A_{d^2,a}(y)\le \sqrt{y/3}d^{-2}+1$.
This implies the upper bound $A_{d^2}(y) \le \sqrt{y/3}d^{-2} \rho(d^2)
+\rho(d^2)$  which is better than claimed.

To prove the lower bound let us suppose
$1\le k\le (\sqrt{y/3} - \epsilon y^{1/2})d^{-2} - 1$. 
With our hypothesis on $y$ we find 
\begin{equation*}
  (1+k)d^2 \le \sqrt{\frac y3}- \epsilon \frac{y}{\sqrt
    y}
 \le 
\sqrt{\frac y3}- \frac{4p^{2n+1}}{\sqrt{3y}}
= \sqrt{\frac y3} \left(1-\frac{4p^{2n+1}}{y}\right).
\end{equation*}
The expression in parentheses lies in $(0,1)$ and so
$(1+k)d^2
\le\sqrt{ y/3}
 \sqrt{1-{4p^{2n+1}}/{y}}$.
If $x=a+kd^2$, then 
$1\le x \le (1+k)d^2 \le \sqrt{(y-4p^{2n+1})/3}$. 
Squaring this expression and rearranging yields 
$3x^2+4p^{2n+1}\le y$ and so this $x$ contributes to $A_{d^2,a}(y)$. 
The number of $k$ involved is at least
$ (\sqrt{y/3}-\epsilon y^{1/2})d^{-2}-2$.  
This yields the  lower bound for $A_{d^2}(y)$  in the assertion after
summing over $a$. 
\end{proof}

Let us assume from now on that $y \ge 4p^{2n+1}\epsilon^{-1}$ as in the last lemma. 

Let $e\in\IN$.
If $\ell\not\in\{2,3,p\}$ is a prime,
then Hensel's Lemma implies the equality in $\rho(\ell^e) =
\rho(\ell)\le 2$. Moreover, $\rho(3^e)=0$ is immediate
since $p\not=3$. 
Elementary considerations yield $\rho(2^e)\le 8$.
Finally, suppose  $f(a)\equiv 0 \mod p^e$ with $a\in\IZ$.
If $e\ge 2n+2$ then $p^{2n+1}|x^2$ and so $p^{2n+2}|x^2$ which
 contradicts (\ref{eq:polyf}). So $\rho(p^e) = 0$.
If $e\le 2n+1$, then $\rho(p^e)\le p^{e/2}$ for even $e$ and
$\rho(p^e)\le p^{(e-1)/2}$ for odd $e$. Therefore,
 $\rho(p^e)\le p^{e/2}$ in any case. 
The function $\rho$ is  multiplicative  by the Chinese Remainder
Theorem, so $\rho(m) \le 2^{2+\omega(m)} p^{\ord_p(m)/2}$ for all $m\in\IN$.
It is well-known that for any $\delta>0$ we have
$2^{\omega(m)}\le c(\delta)m^{\delta}$ where $c(\delta)$ depends only on $\delta$.

\begin{lemma}
\label{lem:estimates2}
The following estimates hold true. 
  \begin{enumerate}
\item  [(i)] We have
  \begin{equation*}
    \sum_{1\le d\le y^{1/3}}
\rho(d^2) = 
O\left(y^{5/12} \right)
\quad\text{and}\quad
    \sum_{1\le d \le y^{1/3}}
\frac{\rho(d^2)}{d^2} = 
O(1).
  \end{equation*}
  \item [(ii)] We have
    \begin{equation*}
    \sum_{d>y^{1/3}}
\frac{\rho(d^2)}{d^2} = 
O(y^{-1/4}).      
    \end{equation*}
\item [(iii)]
  Let $k\in\IN$, then
  \begin{equation}
\label{eq:units}
    \#\{(x,d)\in\IN^2;\,\, 3x^2+4p^{2n+1}=d^2k \le y\} = 
O((\log y)^2). 
  \end{equation}
  \end{enumerate}
\end{lemma}
\begin{proof}
We abbreviate $t=y^{1/3}$ and set
 $I_s = \sum_{1\le d\le t}
\rho(d^2)d^{-s}$ for $s\in \{0,2\}$.
Then
$I_s
= O(\sum_{e\ge 0} p^e \sum_{\atopx{1\le d\le t}
{\ord_p(d)=e}}d^{1/5-s})$ thus $I_s
=
O(\sum_{e\ge 0} p^{e(1+1/5-s)}  \sum_{1\le k\le t/p^e}
k^{1/5-s})$
with $d=p^ek$ in the final sum.
This final sum vanishes if $p^e>t$. 
If $t\ge p^e$ and $s=0$, it is at most $(t/p^e)^{6/5}$.
Using $t=y^{1/3} \ge (p^{2n+1})^{1/3}\ge p$ we get
$I_0
= O(t^{6/5} (\log t)/(\log p))$. This implies
 the first half of (i).
 If $s=2$ then said final sum is 
at most $\sum_{k\ge 1} k^{-9/5}$.
So $I_2  = O (\sum_{e\ge 0} p^{-4e/5})$.
Now $\sum_{e\ge 0}p^{-4e/5} = 1/(1-p^{-4/5})=O(1)$ 
 gives $I_2=O(1)$, the second half of (i). 

For (ii) we use $\rho(p^{2n+2})=0$ and estimate
\begin{equation*}
 I=\sum_{d>t} \frac{\rho(d^2)}{d^2} = O\left(\sum_{e=0}^{2n+1}
p^e\sum_{\atopx{d>t}{\ord_p(d) = e}}
 \frac{1}{d^ {2-1/5}}\right)\quad\text{so}\quad 
 I
=O\left(\sum_{e=0}^{2n+1}p^ {e(1/5-1)}
{\sum_{k>t/p^e}}k^{1/5-2}\right).
\end{equation*}
We compare the final sum with the corresponding integral and obtain
$I =
O(\sum_{e=0}^{2n+1}p^{-4e/5}(t/p^e)^{-4/5})$.
This sum simplifies to 
$\sum_{e=0}^{2n+1}t^{-4/5} = (2n+2)t^{-4/5}$ and
part (ii) follows since  $y\ge p^n$.

Let us now assume $(x,d)$ is in the set on the left side of (\ref{eq:units}).
Then $3 \nmid k$ since $p\not=3$. 
We write $3k=k's^2$ with $s\in\IN$, $k'\in\IN$ square-free, and
 $3|k'$. 
We set $d'= ds, x'=3x$ and note that
$x'^2 -d'^2 k' = -12p^{2n+1}$. 
In other words, the norm of $z'=x'-d'\sqrt{k'}$ is $-12p^{2n+1}$
as an element of the real quadratic field $F=\IQ(\sqrt{k'})$. 
Say $\mathcal{O}_F$ is the ring of integers of this field. 
It contains $z'$ and 
the principal ideal $z'\mathcal{O}_F$ divides
$12p^{2n+1}\mathcal{O}_F$. 
There are  $O(n)$ possibilities for the ideal
$z'\mathcal{O}_F$ as  $p\mathcal{O}_F$ has at most $2$ prime ideal
factors. But $p^n\le y$ and so the number of $z'\mathcal{O}_F$ is also
$O(\log y)$.

However, distinct elements from our original set may lead to the same
ideal. Indeed, the corresponding elements $z',z''$ could be
associated. In this case $z'' = \pm \eta^N z'$ for some $N\in\IZ$ where $\eta>1$ is
the fundamental unit of $F$. We must count how often this happens. 
Taking the exponential absolute
 Weil height $\Height{\cdot}$ and using its basic properties  gives
\begin{equation*}
  \Height{\eta}^{|N|} = \Height{\pm\eta^N}=\Height{z''/z'} \le \Height{z'}\Height{z''}. 
\end{equation*}
From the definition of $z'$  we find
$\Height{z'}\le  2\Height{x'}\Height{d'\sqrt{k'}}
= 2x' d'\sqrt{k'}
\le 6y$. The same inequality holds for $\Height{z''}$ and thus
$|N|\log\Height{\eta} \le 2\log(6y)$. 
But $\log\Height{\eta}>0$ and since $\eta$ is in an quadratic number
field, its logarithmic height is bounded from below by a
positive absolute constant. 
We derive $|N|=O(\log y)$. The number of possibilities
for the pair $(x,d)$ is thus $O((\log y)^2)$. 
\end{proof}

\begin{lemma}
\label{lem:Nyestimate}
  We have
  \begin{equation*}
    N(y) = c(p,n) \sqrt{\frac y3} + O(y^{5/12} + \epsilon y^{1/2})
  \end{equation*}
where 
$c(p,n)= \sum_{d=1}^\infty
\mu(d)\frac{\rho(d^2)}{d^2}>1/7$.
\end{lemma}
\begin{proof}
We split $N(y)$ from (\ref{eq:Ny}) up into
$\sum_{1\le d \le y^{1/3}} \mu(d)A_{d^2}(y) 
+ \sum_{y^{1/3}<d \le y^{1/2}}\mu(d)A_{d^2}(y)$.
The absolute value of the second sum is at most 
\begin{equation*}
 \sum_{y^{1/3}< d\le y^{1/2}} A_{d^2}(y)
= \#\{(x,d)\in \IN^2;\,\, 3x^2+4p^{2n+1}=d^2k\le y \text{ and
  $d>y^{1/3}$ for some }k\in \IN\}. 
\end{equation*}
Any $k$ as above must satisfy $k\le y^{1/3}$.
So the second sum is  $O(y^{1/3}(\log y)^2)$ by part (iii) of the previous lemma.

We use  (\ref{eq:Adestimate})  and  
 Lemma \ref{lem:estimates2}(i) to handle the first sum  and obtain
\begin{alignat*}1
  N(y) 
&=
\sqrt{\frac y3}\sum_{1\le d\le y^{1/3}} \mu(d)
\frac{\rho(d^2)}{d^2}
+  O(y^{5/12}+\epsilon y^{1/2}).
\end{alignat*}
Next we want to replace the sum on the right to a sum over all $d$.
Doing this introduces an error $O(y^{1/4})$
by Lemma \ref{lem:estimates2}(ii). So
\begin{equation*}
N(y) = c(p,n)\sqrt{\frac y3}+ 
O(y^{5/12}+ \epsilon y^{1/2})
\end{equation*}
 with $c(p,n)$ as in the hypothesis.

The function $d\mapsto \mu(d)\rho(d^2)$ is multiplicative.
Hence $c(p,n)$ is represented by the 
 Euler product 
$c(p,n) = \prod_{\ell} (1-\rho(\ell^2)/\ell^{2})$.
We already know that $\rho(\ell^2) = \rho(\ell)\le 2$ if
$\ell\not\in\{2,3,p\}$ and $\rho(3^2)=0$. 
So $\rho(\ell^2)/\ell^2 \le 1/\ell^{3/2}$ except possibly
for $\ell=2$ or $p$. 
One readily checks $\rho(4)=2$ and we know $\rho(p^2)\le p$. 
So
\begin{equation*}
  c(p,n) \ge 
\frac 12\frac{p-1}{p}
\prod_{\ell\not=2,p}\left(1-\frac{1}{\ell^{3/2}}\right)
\ge \frac 25 \prod_{\ell}\left(1-\frac{1}{\ell^{3/2}}\right)
\end{equation*}
because $p\ge 5$. Elementary estimates
now yield $c(p,n)> 1/7$. 
\end{proof}

\begin{proof}[Proof of Proposition \ref{prop:approximate}]
By Lemma \ref{lem:Nyestimate}  and since $7\sqrt{3}<13$
we may fix   $\epsilon > 0$  
such that $N(y) \ge \sqrt{y} /13  + O(y^{5/12})$.
So there is an integer $x$ such that 
 $3x^2+4p^{2n+1}$  is square-free and at most $c'
p^{2n+1}$ where $c'>0$ is absolute. The argument concludes as in the final paragraph
of Section
\ref{sec:quaternion}. 
\end{proof}

\def\cprime{$'$}
\providecommand{\bysame}{\leavevmode\hbox to3em{\hrulefill}\thinspace}
\providecommand{\MR}{\relax\ifhmode\unskip\space\fi MR }
\providecommand{\MRhref}[2]{%
  \href{http://www.ams.org/mathscinet-getitem?mr=#1}{#2}
}
\providecommand{\href}[2]{#2}
\bibliographystyle{amsplain}

\vfill
\address{
\noindent
Philipp Habegger,
Johann Wolfgang Goethe-Universit\"at,
Robert-Mayer-Str. 6-8,
60325 Frankfurt am Main,
Germany,
{\tt habegger@math.uni-frankfurt.de}
}
\bigskip
\hrule
\medskip

 %% \noindent {\tt \jobname.tex} created: January 6, 2012. Latest update: \today.

\end{document}